\documentclass[12pt,a4paper]{article}
\usepackage{amssymb,amsmath}
\usepackage{hyperref}
\usepackage{amsthm}
\usepackage{ulem}
\newtheorem*{theorem}{Theorem}
\usepackage{chngcntr}
\usepackage{color}
\usepackage{graphicx}
\usepackage[dvipsnames]{xcolor}
\usepackage[symbol]{footmisc}
\usepackage[english,ngerman]{babel}
\usepackage[top=2.0cm,bottom=2.0cm,left=2.0cm,right=2.0cm]{geometry}

\begin{document}
\begin{center}
\LARGE{\textbf{A fixed point iteration method
\\for the arctangent with any odd order
\\of convergence based on sine and cosine}}
\[\]
\[\]
\large
{Alois Schiessl}
\footnote[1]{{University of Regensburg}\,\;E-Mail: \texttt{aloisschiessl@posteo.de}}
\selectlanguage{ngerman}
\centerline{ }
\centerline{\today}
\end{center}
\selectlanguage{english}
\[\]
\[\]
\[\]
\begin{abstract}
In this paper, we present a fixed point method for the arctangent based on sine and cosine. Let $t\in \mathbb{R}^{+}$ and $P\in \mathbb{N}$. We define:
\[T\left(x\right)=x-\sum_{k=1}^{P}\,\frac{\left(-1\right)^{k-1}}{2\,k-1}
\left(\frac
{\sin\!\left(x\right)-t\cos\!\left(x\right)}
{\cos\!\left(x\right)+t\sin\!\left(x\right)}
\right)^{2\,k-1}.\]
For every initial value $x_0$ sufficiently close to $\arctan\left(t\right)$, the sequence
\[x_{n+1}=T\left(x_{n}\right)\;;\,n=0,1,\ldots\]
is converging to $\arctan\left(t\right)$ with order of convergence exactly $\left(2\,P+1\right)$. The computational test we performed demonstrates the efficiency of the method.
\selectlanguage{ngerman}
\[\]
\[\textbf{Zusammenfassung}\]
In dieser Abhandlung stellen wir ein Fixpunktverfahren zur Berechnung des arcustangens auf Basis von sinus und cosinus  vor. Es sei $t\in \mathbb{R}^{+}$ und $P\in\mathbb{N}$. Wir definieren:
\[T\left(x\right)=x-\sum_{k=1}^{P}\,\frac{\left(-1\right)^{k-1}}{2\,k-1}
\left(\frac
{\sin\!\left(x\right)-t\cos\!\left(x\right)}
{\cos\!\left(x\right)+t\sin\!\left(x\right)}\right)
^{2\,k-1}.\]
Für jeden Startwert $x_0$ hinreichend nahe bei $\arctan\left(t\right)$ konvergiert die Folge
\[x_{n+1}=T\left(x_{n}\right)\;;\,n=0,1,\ldots\]
gegen $\arctan\left(t\right)$ mit Konvergenzordnung genau $\left(2\,P+1\right)$. Anhand einer praktischen Berechnung von $\frac{\pi}{4}$ zeigen wir die Effizienz des Verfahrens.
\[\text{Deutsche Version ab Seite 17}\]
\end{abstract}
\section{Introduction and main result}
We start with the $\arctan$ power series \cite{AS}
\[
\arctan\left(u\right)
=u-\frac{u^3}{3}+\frac{u^5}{5}-\frac{u^7}{7}+-\dots
=\sum_{k=1}^{\infty}\,\frac{\left(-1\right)^{k-1}}{2\,k-1}\,u^{\,2\,k-1}\,.
\]
The power series is absolutely convergent for $\left|u\right|< 1$. We only use the first $P$ terms: 
\[
u-\frac{u^3}{3}+\frac{u^5}{5}-\frac{u^7}{7}+-\ldots+\frac{\left(-1\right)^{P-1}}{2\,P-1}\,u^{2\,P-1}
=\sum_{k=1}^{P}\frac{\left(-1\right)^{k-1}}{2\,k-1}\,u^{2\,k-1}\,.
\]
Let $t\in\mathbb{R}^{+}$ and substituting
\[u=\frac{\sin\!\left(x\right)-t\cos\!\left(x\right)}{\cos\!\left(x\right)+t\sin\!\left(x\right)}\]
leads to the finite sum
\[
\sum_{k=1}^{P}\,\frac{\left(-1\right)^{k-1}}{2\,k-1}
\left(\frac
{\sin\!\left(x\right)-t\cos\!\left(x\right)}
{\cos\!\left(x\right)+t\sin\!\left(x\right)}
\right)^{2\,k-1}.
\]
With this sum we define a fixed point function that enables the calculation of $\arctan\left(t\right)$ with any odd order of convergence. Since $\arctan\left(-t\right)$ = $-\arctan\left(t\right)$, it is sufficient considering only positiv values.
$ \\ $
$ \\ $
Now let's formulate our statement:
\begin{theorem}
$ \\ $
Let $t\in{\mathbb{R^{+}}}$ and $P\in{\mathbb{N}}$. We define the function
\[T:\;\;[\,0\,,\frac{\pi}{2}]\;\;\rightarrow\mathbb{R}\;;\]
\[\quad x\mapsto T\left(x\right)\;;\]
\[T\left(x\right)=x-\sum_{k=1}^{P}\,\frac{\left(-1\right)^{k-1}}{2\,k-1}
\left(\frac{\sin\!\left(x\right)-t\cos\!\left(x\right)}{\cos\!\left(x\right)+t\sin\!\left(x\right)}\right)
^{2\,k-1}.\]
Then, the following statements hold, for all $t\in \mathbb{R}^+:$
\begin{itemize}
\item[(a)]
$T\left(x\right)$ is a fixed point function with $\arctan\left(t\right)$ as fixed point;\\ that means $T\bigl(\arctan\left(t\right)\bigr)=\arctan\left(t\right)\,.$
\item[(b)]
For the derivatives at the fixed point applies $:$
\begin{align*}
T^{\left(k\right)}{\bigl(\arctan\left(t\right)\bigr)}=\left\{
\begin{aligned}
0\quad\qquad\qquad\qquad\quad 0\le &\,k\leq 2\, P
\\
\\
\left(-1\right)^{P}\left(2\cdot P\right)\,!
\qquad\qquad\quad k=\,&2\, P+1
\\ 
\end{aligned} 
\right.
\end{align*}
\item[(c)] For every initial value $x_0$ sufficiently close to $\arctan\left(t\right)$, the sequence
\begin{align*}
x_{n+1}&=T\left(x_{n}\right)\,,\;n=0,1,\,\ldots
\end{align*}
is converging to $\arctan\left(t\right)$ with order of convergence exactly $\left(2\,P+1\right)\,.$
\end{itemize}
\end{theorem}
\section{Proof of the Theorem}
In this section, we prove the statements of the theorem.
In order to simplify the notation for
\[
T\left(x\right)=x-\sum_{k=1}^{P}\,\frac{\left(-1\right)^{k-1}}{2\,k-1}
\left(\frac
{\sin\!\left(x\right)-t\cos\!\left(x\right)}
{\cos\!\left(x\right)+t\sin\!\left(x\right)}
\right)^{2\,k-1},
\]
we define the function:
$\\ $
Let $t\in \mathbb{R^+}$, 
\[f:\;\;[0,\frac{\pi}{2}]\;\;\rightarrow\mathbb{R}\;;\]
\[f\left(x\right)=\frac
{\sin\!\left(x\right)-t\cos\!\left(x\right)}
{\cos\!\left(x\right)+t\sin\!\left(x\right)}\,.\]
The fixed point function is then easier to write as
\[
T\left(x\right)=x-\sum_{k=1}^{P}\,\frac{\left(-1\right)^{k-1}}{2\,k-1}f\left(x\right)^{2\,k-1}
\]
\subsection{Proof of fixed point}
We must prove that the following statement is true at the fixed point $\arctan\left(t\right)$\;:
\[
T\bigl(\arctan\left(t\right)\bigr)=\arctan\left(t\right).
\]
It suffices to prove that the following relation holds
\[
f\bigl(\arctan\left(t\right)\bigr)=\frac{\sin\!\bigl(\arctan\left(t\right)\bigr)-t\cos\!\bigl(\arctan\left(t\right)\bigr)}{\cos\!\bigl(\arctan\left(t\right)\bigr)+t\sin\!\bigl(\arctan\left(t\right)\bigr)}=0\,.
\]
From a mathematical handbook \cite{SO}, we obtain the following statements\,:
\[\sin\,\bigl(\arctan\left(t\right)\bigr)=\frac{t}{\sqrt{1+t^2}}\;\;;\;t\in \mathbb{R}\;;\]
\[\cos\,\bigl(\arctan\left(t\right)\bigr)=\frac{1}{\sqrt{1+t^2}}\;\;;\;t\in \mathbb{R}\;;\]
Plugging in gives
\[
f\bigl(\arctan\left(t\right)\bigr)
=\frac{\frac{t}{\sqrt{1+t^2}}-t\frac{1}{\sqrt{1+t^2}}}{\frac{1}{\sqrt{1+t^2}}+t\frac{t}{\sqrt{1+t^2}}}
=\frac{\frac{t}{\sqrt{1+t^2}}-\frac{t}{\sqrt{1+t^2}}}{\frac{1}{\sqrt{1+t^2}}+\frac{t^2}{\sqrt{1+t^2}}}
=\frac{0}{\frac{1}{\sqrt{1+t^2}}+\frac{t^2}{\sqrt{1+t^2}}}=0\,.
\]
From $f\bigl(\arctan\left(t\right)\bigr)=0$ follows directly that
\[
T\bigl(\arctan\left(t\right)\bigr)=\arctan\left(t\right)-\sum_{k=1}^{P}\,\frac{\left(-1\right)^{k-1}}{2\,k-1}
{\underbrace{f\bigl(\arctan\left(t\right)\bigr)}_{=0}}^{2\,k-1}=\arctan\left(t\right)\,.
\]
This proves the fixed point property.
\subsection{Proof of the derivatives}
Next, we turn to the derivatives. This requires a little bit more work.
\subsubsection{First Derivative}
Let's start with the first derivative. We use the fixed point function in the representation
\[
T\left(x\right)=x-\sum_{k=1}^{P}\,\frac{\left(-1\right)^{k-1}}{2\,k-1}f\left(x\right)^{2\,k-1}
\]
with
\[
f\left(x\right)=\frac{\sin\!\left(x\right)-t\cos\!\left(x\right)}{\cos\!\left(x\right)+t\sin\!\left(x\right)}\,.
\]
Differentiating the fixed point function yields
\begin{align*}
T\,'\left(x\right)
&=\left(x-\sum_{k=1}^{P}\,\frac{\left(-1\right)^{k-1}}{2\,k-1}f\left(x\right)^{2\,k-1}\right)'
\\
&=1-\sum_{k=1}^{P}\left(-1\right)^{k-1}f\left(x\right)^{2\,k-2}f\,'\left(x\right)\,.
\\
&=1+\sum_{k=1}^{P}\left(-1\right)^{k}f\left(x\right)^{2\,k-2}f\,'\left(x\right)\,.
\end{align*}
We need the derivative of
\[
f\left(x\right)=\frac{\sin\!\left(x\right)-t\cos\!\left(x\right)}{\cos\!\left(x\right)+t\sin\!\left(x\right)}\,.
\]
We leave this task to the computer algebra system Maple and trust in the result:
\[
f\,'\left(x\right)=1+\left(\frac
{\sin\!\left(x\right)-t\cos\!\left(x\right)}
{\cos\!\left(x\right)+t\sin\!\left(x\right)}
\right)^2=1+f\left(x\right)^2.
\]
Substituting this into the first derivative, we get
\begin{align*}
T\,'\left(x\right)
&=1+\sum_{k=1}^{P}\left(-1\right)^{k}f\left(x\right)^{2\,k-2}\left(1+f\left(x\right)^2\right)
\\
&=1+\sum_{k=1}^{P}\left(-1\right)^{k}f\left(x\right)^{2\,k-2}
+\sum_{k=1}^{P}\left(-1\right)^{k}f\left(x\right)^{2\,k}.
\end{align*}
To make the matter more transparent, we will refer to the first sum as
\begin{align*}
s_1=1+\sum_{k=1}^{P}\left(-1\right)^{k}f\left(x\right)^{2\,k-2}
\end{align*}
and the second sum as
\[
s_2=\sum_{k=1}^{P}\left(-1\right)^{k}f\left(x\right)^{2\,k}
\]
The first derivative can then be written simply as
\[
T\,'\left(x\right)=s_1+s_2\,.
\]
Next step is splitting the sum $s_1$ as follows:
\begin{align*}
s_1
&=1+\sum_{k=1}^{1}\left(-1\right)^{k}f\left(x\right)^{2\,k-2}
+\sum_{k=2}^{P}\left(-1\right)^{k}f\left(x\right)^{2\,k-2}
\\
&=1-1
+\sum_{k=2}^{P}\left(-1\right)^{k}f\left(x\right)^{2\,k-2}
\\
&=
\sum_{k=2}^{P}\left(-1\right)^{k}f\left(x\right)^{2\,k-2}.
\end{align*}
Now shifting the summation index leads to
\begin{align*}
s_1
&=\sum_{k=1}^{P-1}\left(-1\right)^{k+1}f\left(x\right)^{2\,k}
\\
=&-\sum_{k=1}^{P-1}\left(-1\right)^{k}f\left(x\right)^{2\,k}\,.
\end{align*}
Next, let's look at $s_2$. Also splitting the sum and simplifying gives a more simplified form:
\begin{align*}
s_2&=\sum_{k=1}^{P-1}\left(-1\right)^{k}f\left(x\right)^{2\,k}
+\sum_{k=P}^{P}\left(-1\right)^{k}f\left(x\right)^{2\,k}
\\
&=\sum_{k=1}^{P-1}\left(-1\right)^{k}f\left(x\right)^{2\,k}
+\left(-1\right)^{P}f\left(x\right)^{2\,P}.
\end{align*}
Up to now we have
\begin{align*}
s_1
&=-\sum_{k=1}^{P-1}\left(-1\right)^{k}f\left(x\right)^{2\,k}
\\
s_2
&=\;\;\;\sum_{k=1}^{P-1}\left(-1\right)^{k}\bigl(f\left(x\right)\bigr)^{2\,k}
+\left(-1\right)^{P}f\left(x\right)^{2\,P}\,.
\end{align*}
When used, we get
\begin{align*}
T\,'\left(x\right)
=&\;\;\;s_1+s_2
\\
=&-\sum_{k=1}^{P-1}\left(-1\right)^{k}f\left(x\right)^{2\,k}
\\
&+\sum_{k=1}^{P-1}\left(-1\right)^{k}f\left(x\right)^{2\,k}
+\left(-1\right)^{P}f\left(x\right)^{2\,P}
\\
&=\left(-1\right)^{P}f\left(x\right)^{2\,P}\,.
\end{align*}
The sums are vanishing and it remains the simple term:
\[T\,'\left(x\right)
=\left(-1\right)^{P}f\left(x\right)^{2\,P}
=\left(-1\right)^{P}\left(\frac
{\sin\!\left(x\right)-t\cos\!\left(x\right)}
{\cos\!\left(x\right)+t\sin\!\left(x\right)}
\right)^{2\,P}\,.
\]
Now plugging in the fixed point $\arctan\left(t\right)$ gives us
\[
T\,'\bigl(\arctan\left(t\right)\bigr)=\left(-1\right)^{P}\left(\frac
{\sin\bigl(\arctan\left(t\right)\bigr)-t\cos\bigl(\arctan\left(t\right)\bigr)}
{\cos\bigl(\arctan\left(t\right)\bigr)+t\sin\bigl(\arctan\left(t\right)\bigr)}
\right)^2\,.
\]
We have already calculated the term in brackets in the fixed point proof:
\[\frac
{\sin\bigl(\arctan\left(t\right)\bigr)-t\cos\bigl(\arctan\left(t\right)\bigr)}
{\cos\bigl(\arctan\left(t\right)\bigr)+t\sin\bigl(\arctan\left(t\right)\bigr)}=0\,.\]
Plugging in, finally results in
\[T\,'\bigl(\arctan\left(t\right)\bigr)=0\,.\]
Up to now, we have the particular results at the fixed point
\[T\,\bigl(\arctan\left(t\right)\bigr)=\arctan\left(t\right)\,;\]
\[T\,'\bigl(\arctan\left(t\right)\bigr)=0\,.\]
\subsubsection{Higher derivatives}
We still have to prove that the following statement holds at the fixed point $\arctan\left(t\right):$
\begin{align*}
T^{\left(k\right)}{\bigl(\arctan\left(t\right)\bigr)}=\left\{
\begin{aligned}
0\;\quad\qquad\qquad\;\; 1\leq &\,k\le 2\,P
\\
\\
\left(-1\right)^{P}\left(2\,P\right)\,!
\qquad\quad k=&2\,P+1
\\ 
\end{aligned} 
\right.
\end{align*}
We must differentiate $\left(2\,P+1\right)$ times $T\left(x\right)$ and then plugging in $\arctan\left(t\right)$. Using the first derivative $T\,'\left(x\right)$, the task simplifies to 
\begin{align*}
\left.T^{\left(2\,P+1\right)}\left(x\right)\right|_{x=\arctan\left(t\right)}
=\left.\left(T\,'\right)^{\left(2\,P\right)}\left(x\right)\right|_{x=\arctan\left(t\right)}
=\left.\left(\left(-1\right)^{P}f\left(x\right)^{2\,P}\right)^{\left(2\,P\right)}\right|_{x=\arctan\left(t\right)}\,.
\end{align*}
On the right-hand side, we need at $\arctan\left(t\right)$ the $\left(2\,P\right)^{th}$ derivative of $f\left(x\right)^{2\,P}$. Remembering
\[f\bigl(\arctan\left(t\right)\bigr)=0\]
and calculating
\[f'\bigl(\arctan\left(t\right)\bigr)=1+f\left(\arctan\left(t\right)\right)^2=1\,.\]
Then, we proceed as follows: We expand $f\left(x\right)$ into a Taylor series of degree 1 at the fixed point $\arctan\left(t\right)$ with Lagrange remainder:
\[
f\left(x\right)
=f\bigl(\arctan\left(t\right)\bigr)
+f'\bigl(\arctan\left(t\right)\bigr)\bigl(x-\arctan\left(t\right)\bigr)
+\frac{f\,''\left(\xi\right)}{2}\bigl(x-\arctan\left(t\right)\bigr)^{2}\,.
\]
where $\xi$ is between $x$ and $\arctan\left(t\right)$.
$\\ $
Since, in accordance with $f\bigl(\arctan\left(t\right)\bigr)=0$ and $f'\bigl(\arctan\left(t\right)\bigr)=1$, we obtain
\[
f\left(x\right)
=\bigl(x-\arctan\left(t\right)\bigr)+\frac{f\,''\left(\xi\right)}{2}\bigl(x-\arctan\left(t\right)\bigr)^{2}\,.\]
Factoring out $\bigl(x-\arctan\left(t\right)\bigr)$ yields
\[
f\left(x\right)
=\bigl(x-\arctan\left(t\right)\bigr)\left(1+\frac{f\,''\left(\xi\right)}{2}\bigl(x-\arctan\left(t\right)\bigr)\right)\,.
\]
Next raising the equation to the $\left(2\,P\right)^{\,th}$ power:
\[
f\left(x\right)^{2\,P}
=\bigl(x-\arctan\left(t\right)\bigr)^{2\,P}
\left(1+\frac{f{\,''}\left(\xi\right)}{2}\bigl(x-\arctan\left(t\right)\bigr)\right)^{2\,P}\,.
\]
We calculate the higher derivatives of $f\left(x\right)^{2\,P}$ by using the Leibniz rule \cite{SO}.
Leibniz rule states that if $u\left(x\right)$ and $v\left(x\right)$ are $n$-times differentiable functions, then the product $u\left(x\right)\cdot v\left(x\right)$ is also $n$-times differentiable and its $n^{\,th}$ derivative is given by the sum
\[
\bigl(u\left(x\right)\cdot v\left(x\right)\bigr)^{\left(n\right)}
=\sum\limits_{k=0}^{n}\binom{n}{k}\,u\left(x\right)^{\left(k\right)}\cdot v\left(x\right)^{\left(n-k\right)}\,.
\]
Here we want to understand $u\left(x\right)^{\left(0\right)}=u\left(x\right)$ and $v\left(x\right)^{\left(0\right)}=v\left(x\right)\,.$
$\\ $
$\\ $
In our case we have:
\[
n=2\,P\;;
\]
\[
u\left(x\right)=\bigl(x-\arctan\left(t\right)\bigr)^{2\,P}\;;
\]
\[
v\left(x\right)=\left(1+\frac{f\,''\left(\xi\right)}{2}\bigl(x-\arctan\left(t\right)\bigr)\right)^{2\,P}\;;
\]
Thus we obtain
\[
f\left(x\right)^{2\,P}=
\underbrace{{\bigl(x-\arctan\left(t\right)\bigr)}^{2\,P}}_{u\left(x\right)}
\cdot
\underbrace{{{\left(1+\frac{f\,''{\left(\xi\right)}}{2}\bigl(x-\arctan\left(t\right)\bigr)\right)}^{2\,P}}}_{v\left(x\right)}\,.
\]
Now applying Leibniz rule gives
\begin{align*}
\left(f\left(x\right)^{2\,P}\right)^{\left({2\,P}\right)}
&=\sum\limits_{k=0}^{{2\,P}}\binom{{2\,P}}{k}
\,\left(\bigl(x-\arctan\left(t\right)\bigr)^{2\,P}\right)^{\left(k\right)}
\\
&\times
\left(\left(1+\frac{f\,''\left(\xi\right)}{2}\bigl(x-\arctan\left(t\right)\bigr)\right)^{{2\,P}}\right)^{\left({2\,P-k}\right)}
\end{align*}
We first focus on the sequence of derivatives
\[
u\left(x\right)^{\left(k\right)}
=\left(\bigl(x-\arctan\left(t\right)\bigr)^{{2\,P}}\right)^{\left\{k\right\}},k=0,1,\ldots,{2\,P}\,.
\]
From a mathematical handbook \cite{SO}, we derive the following formula:
\begin{align*}
\Bigl(\bigl(x-\arctan\left(t\right)\bigr)^{2\,P}\Bigr)^{\left(k\right)}
=\frac{\left(2\,P\right)\,!}{\left(2\,P-k\right)\,!}\bigl(x-\arctan\left(t\right)\bigr)^{2\,P-k}\,,k=0,1,\ldots,{2\,P}.
\end{align*}
$\\ $
In particular, we note that for $k<2\,P$ in
\[
u\left(x\right)^{\left(k\right)}=
\left(\bigl(x-\arctan\left(t\right)\bigr)^{{2\,P}}\right)^{\left(k\right)}
\]
the factor
$\bigl(x-\arctan\left(t\right)\bigr)$ always occurs and thus for $x=\arctan\left(t\right)$ the terms are vanishing. Only the case $k=\left(2\,P\right)$ leads to a non-zero result, namely the product
\[
{\left.u{{\left(x\right)}^{\left(2\,P\right)}}\right|}_{x=\arctan\left(t\right)}=\left(2\,P\right)\,!\,.
\]
From the sum
\[
\sum\limits_{k=0}^{{2\,P}}\binom{{2\,P}}{k}
\,\left(\bigl(x-\arctan\left(t\right)\bigr)^{2\,P}\right)^{\left(k\right)}
\cdot
\left(\left(1+\frac{f\,''{\left(\xi\right)}}{2}\bigl(x-\arctan\left(t\right)\bigr)\right)^{{2\,P}}\right)^{\left({2\,P-k}\right)},
\]
it only rests the term with $k=\left(2\,P\right)$ at the fixed point $\arctan\left(t\right)$. We can evaluate
\begin{align*}
{\left.{{\left(f{{\left(x\right)}^{2\,P}}\right)}^{\left(2\,P\right)}}\right|}_{x=\arctan\left(t\right)}
&=\underbrace{\sum\limits_{k={2\,P}}^{{2\,P}}\binom{{2\,P}}{{2\,P}}\left.\left(\bigl(x-\arctan\left(t\right)\bigr)^{2\,P}\right)^{\left(2\,P\right)}\right|_{x=\arctan\left(t\right)}}_{=\left(2\,P\right)\,!}
\\
&\times\underbrace{\left.\left(\left(1
+\frac{f\,''\left(\xi\right)}{2}\bigl(x-\arctan\left(t\right)\bigr)\right)^{2\,P}\right)\right|_{x=\arctan\left(t\right)}}_{=1}\,.
\end{align*}
We obtain the simple result
\[
{\left.{{\left(f{{\left(x\right)}^{2\,P}}\right)}^{\left(2\,P\right)}}\right|}_{x=\arctan\left(t\right)}
=\left(2\,P\right)\,!\,.
\]
Our task was the calculation of
\begin{align*}
\left.T\left(x\right)^{\left(2\,P+1\right)}\right|_{x=\arctan\left(t\right)}
=\left.T\,'\left(x\right)^{\left(2\,P\right)}\right|_{x=\arctan\left(t\right)}
=\left.\left(\left(-1\right)^{P}f\left(x\right)^{2\,P}\right)^{\left(2\,P\right)}\right|_{x=\arctan\left(t\right)}\,.
\end{align*}
Now plugging in yields
\[
T\bigl(\arctan\left(t\right)\bigr)^{\left(2\,P+1\right)}
=\left(-1\right)^P\left(2\,P\right)\,!
\]
This completes the proof of the derivations.
\subsection{Proof of convergence}
We do the prove by using Banach's fixed point theorem \cite {WG} in a special form for continuously differentiable functions in ${\mathbb{R}}^{1}\,.$
$ \\ $
$ \\ $
\underline{Banach fixed point theorem}
\\
Let $U\subseteq R$ be a closed subset; furthermore, let $F:U\rightarrow R$ be a mapping with the following properties\[
F\left(U\right)\subseteq U \,.
\]
There exists $0<L<1$ such that
\[
\left|F\,'\left(x\right)\right| \leq L\,,\; \text{for\,\;all} \;x\in U \,.
\]
Then the following statements hold:
\begin{itemize}
\item[(a)]
There is exactly one fixed point ${{x}^{*}}\in U$ of $F\,.$
\item[(b)]
For every initial value $x_0\in U$, the sequence
$x_{n+1}=F\left(x_{{n}}\right)$
converges to $x^*\,.$
\end{itemize}
$\\ $
\underline{Proof of convergence} 
$\\ $
We have already proven that $T\bigl(\arctan\left(t\right)\bigr)=\arctan\left(t\right)$ and $T\,'\bigl(\arctan\left(t\right)\bigr)=0$ hold true. Then, for reasons of continuity, there exists a neighbourhood at the fixed point $\arctan\left(t\right)$
\[U_{\delta}=\left\{x\in\mathbb{R}:\left|x-\arctan\left(t\right)\right|\leq\delta\right\}\]
in which the following applies: $T\left(U_\delta\right)\subseteq U_\delta$ and $\left|T\,'\left(x\right)\right|<L<1$ for all $x\in U_{\delta}$.
Thus, the conditions for the fixed point theorem are satisfied and convergence follows.
\subsection{Order of convergence}
Finally, we have to proof that the order of convergence is exactly $\left(2\,P+1\right)$.
We expand $T\left(x\right)$ into a Taylor series around the fixed point $\arctan\left(t\right)\,:$
\begin{align*}
T\left(x\right)&=T\bigl(\arctan\left(t\right)\bigr)
\\&
+T\,'\bigl(\arctan\left(t\right)\bigr)\bigl(x-\arctan\left(t\right)\bigr)
+\ldots+T^{\left(2\,P\right)}\bigl(\arctan\left(t\right)\bigr)
\frac{\bigl(x-\arctan\left(t\right)\bigr)^{2\,P}}{\left(2\,P\right)\,!}
\\
&+T^{\left(2\,P+1\right)}\left(\xi\right)
\frac{\bigl(x-\arctan\left(t\right)\bigr)^{2\,P+1}}{\left(2\,P+1\right)\,!}\;,
\end{align*}
where $\xi$ is between $x$ and $\arctan\left(t\right)$. Let us remind $T\bigl(\arctan\left(t\right)\bigr)=\arctan\left(t\right)$ and the differentiation results
\begin{align*}
T^{\left(k\right)}{\bigl(\arctan\left(t\right)\bigr)}=\left\{
\begin{aligned}
0\quad\qquad\qquad\qquad\quad 0\le &k\leq 2\, P
\\
\\
\left(-1\right)^{P}\left(2\cdot P\right)\,!
\qquad\qquad\quad k=&2\, P+1
\\ 
\end{aligned} 
\right.
\end{align*}
The Taylor series thus simplifies to
\begin{align*}
T\left(x\right)
&=\underbrace{T\bigl(\arctan\left(t\right)\bigr)}_{=\arctan\left(t\right)}
\\
&+\underbrace{T\,'\bigl(\arctan\left(t\right)\bigr)\bigl(x-\arctan\left(t\right)\bigr)
+\ldots
+T^{\left(2\,P\right)}\bigl(\arctan\left(t\right)\bigr)
\frac{\bigl(x-\arctan\left(t\right)\bigr)^{2\,P}}{\left(2\,P\right)\,!}}_{=0}
\\
&+T^{\left(2\,P+1\right)}\left(\xi\right)
\frac{\bigl(x-\arctan\left(t\right)\bigr)^{2\,P+1}}{\left(2\,P+1\right)\,!}
\\
&=\arctan\left(t\right)+T\left(\xi\right)^{\left(2\,P+1\right)}
\frac{\bigl(x-\arctan\left(t\right)\bigr)^{2\,P+1}}{\left(2\,P+1\right)\,!}\,.
\end{align*}
Now plugging in $x=x_n$ sufficiently close to $\arctan\left(t\right)$, we have
\begin{align*}
T\left(x_{n}\right)
=\arctan\left(t\right)
+T^{\left(2\,P+1\right)}\left(\xi_{n}\right)
\frac{\left(x_{n}+\arctan\left(t\right)\right)^{2\,P+1}}{\left(2\,P+1\right)\,!}\,.
\end{align*}
Since $T\left(x_n\right)=x_{n+1}$, we obtain
\begin{align*}
x_{n+1}=\arctan\left(t\right)
+T\,^{\left(2\,P+1\right)}\left(\xi_{n}\right)
\frac{\left(x_{n}-\arctan\left(t\right)\right)^{2\,P+1}}{\left(2\,P+1\right)\,!}\,.
\end{align*}
Taking $\arctan\left(t\right)$ to the left side and dividing by $\left(x_{n+1}-\arctan\left(t\right)\right)^{2\,P+1}$ yields
\[
\frac{x_{n+1}-\arctan\left(t\right)}{\left(x_{n}-\arctan\left(t\right)\right)^{2\,P+1}}=
T^{\left(2\,P+1\right)}\left(\xi_{n}\right)
\frac{1}{\left(2\,P+1\right)\,!}\,.
\]
As $x_n\rightarrow \arctan\left(t\right)$, since $ \xi_n$ is trapped between $x_n$ and $\arctan\left(t\right)$ we conclude, by the continuity of $T^{\,\left(P+1\right)}\left(x\right)$ at $\arctan\left(t\right)$, that
\begin{align*}
\underset{n\to \infty}{\mathop{\lim }}\; \frac{x_{n+1}-\arctan\left(t\right)}{\left(x_{n}-\arctan\left(t\right)\right)^{2\,P+1}}
=\underset{{n}\to \infty}{\mathop{\lim }}\;
T^{\left(2\,P+1\right)}\left(\xi_{n}\right)
\frac{1}{\left(2\,P+1\right)\,!}
=T^{\left(2\,P+1\right)}\bigl(\arctan\left(t\right)\bigr)
\frac{1}{\left(2\,P+1\right)\,!}\,.
\end{align*}
We use an earlier result
\[
T^{\left({2\,P+1}\right)}{\bigl(\arctan\left(t\right)\bigr)}
=\left(-1\right)^{P}\left(2\,P\right)\,!\,.
\]
Plugging in gives
\begin{align*}
\underset{n\to\infty}{\mathop{\lim }}\;\;
\frac{x_{n+1}-\arctan\left(t\right)}{\left(x_{n}-\arctan\left(t\right)\right)^{2\,P+1}}
=\left(-1\right)^{P}\frac{\left(2\,P\right)\,!}{\left(2\,P+1\right)\,!}
=\left(-1\right)^{P}\frac{\left(2\,P\right)\,!}{\left(2\,P\right)\,!\left(2\,P+1\right)}
=\frac{\left(-1\right)^{P}}{2\,P+1}\,.
\end{align*}
This shows that convergence is exactly of the order $\left(2\,P+1\right)$, and
\[
\frac{\left(-1\right)^{P}}{2\,P+1}
\]
is the asymptotic error constant.
\section{\texorpdfstring{Practical computation}{}}
Now let’s move on to the application of the fixed point function. Let $t\in\mathbb{R}^{+}$ and $P\in\mathbb{N}$. That gives the fixed point function
\[T\left(x\right)=x-\sum_{k=1}^{P}
\frac{\left(-1\right)^{k-1}}{2\,k-1}
\left(\frac
{\sin\!\left(x\right)-t\cos\!\left(x\right)}
{\cos\!\left(x\right)+t\sin\!\left(x\right)}
\right)^{2\,k-1}
\]
With an initial value $x_0$ sufficiently close to $\arctan\left(t\right)$, we start the iteration:
\[x_{1}=x_{0}-\sum\limits_{k=1}^{P}
\frac{\left(-1\right)^{k-1}}{2\,k-1}
\left(\frac
{\sin\!\left(x_0\right)-t\cos\!\left(x_0\right)}
{\cos\!\left(x_0\right)+t\sin\!\left(x_0\right)}
\right)^{2\,k-1}\,;
\]

\[x_{2}=x_{1}-\sum\limits_{k=1}^{P}
\frac{\left(-1\right)^{k-1}}{2\,k-1}
\left(\frac
{\sin\!\left(x_1\right)-t\cos\!\left(x_1\right)}
{\cos\!\left(x_1\right)+t\sin\!\left(x_1\right)}
\right)^{2\,k-1}\,;
\]

\[x_{3}=x_{2}-\sum\limits_{k=1}^{P}
\frac{\left(-1\right)^{k-1}}{2\,k-1}
\left(\frac
{\sin\!\left(x_2\right)-t\cos\!\left(x_2\right)}
{\cos\!\left(x_2\right)+t\sin\!\left(x_2\right)}
\right)^{2\,k-1}\,;
\]

\[\vdots\]

\[x_{n+1}=x_{n}-\sum\limits_{k=1}^{P}
\frac{\left(-1\right)^{k-1}}{2\,k-1}
\left(\frac
{\sin\!\left(x_n\right)-t\cos\!\left(x_n\right)}
{\cos\!\left(x_n\right)+t\sin\!\left(x_n\right)}
\right)^{2\,k-1}\,.
\]
$ \\ $
If, now $n\rightarrow \infty$, then $x_{n+1}\rightarrow \arctan\left(t\right)$ with order of convergence $\left(2\,P+1\right)$. With unlimited computing time and unlimited storage space, we can calculate $\arctan\left(t\right)$  with unlimited precision. Unfortunately, we do not have anything like that available. We must be content with a finite number of iterations and a finite value of $\arctan\left(t\right)$.
$ \\ $
$ \\ $
We require a termination criterion. For this purpose, let $\epsilon>0\,.$ We terminate the iteration if the following statement holds
\[\left|x_{n+1}-x_{n}\right|<\epsilon\,.\]
\subsection{\texorpdfstring{Machin-Like Formulas for $\frac{\pi}{4}$}{}}
Now, we turn to Machin-like formulas. Machin’s formulas are particular identities for $\frac{\pi}{4}$ as a sum of integer multiples of arctangent having rational arguments.
\[\frac{\pi}{4}=\sum_{k=1}^{n}a_n\arctan\frac{1}{b_n}\;,\;\;a_n\in\mathbb Z\,;\;b_n\in\mathbb N\]
Here, the argument takes the value $t=\frac{1}{q}$, where $q\in \mathbb{N}.$ This allows us to obtain a further representation. The fixed-point function becomes
\begin{align*}
T\left(x\right)
&=x-\sum_{k=1}^{P}\,\frac{\left(-1\right)^{k-1}}{2\,k-1}\left(\frac
{\sin\left(x\right)-\frac{1}{q}\cdot\cos\left(x\right)}
{\cos\left(x\right)+\frac{1}{q}\cdot\sin\left(x\right)}
\right)^{2\,k-1}
\\
&=x-\sum_{k=1}^{P}\,\frac{\left(-1\right)^{k-1}}{2\,k-1}\,\left(\frac
{q\cdot\sin\left(x\right)-\cos\left(x\right)}
{q\cdot\cos\left(x\right)+\sin\left(x\right)}\right)^{2\,k-1};
\end{align*}
From this, we derive the iteration formula:
\[
x_{n+1}=x_{n}-\sum_{k=1}^{P}\,\frac{\left(-1\right)^{k-1}}{2\,k-1}\,\left(\frac
{q\cdot\sin\left(x_{n}\right)-\cos\left(x_{n}\right)}
{q\cdot\cos\left(x_{n}\right)+\sin\left(x_{n}\right)}\right)^{2\,k-1};\;\;n=0,1,\ldots
\]
\subsubsection{1-Term formula}
There is a single formula with one term
\[
\frac{\pi}{4}=\arctan\left(1\right)
\]
So $q=1$, we obtain
\[
T\left(x\right)=x-\sum_{k=1}^{P}\,\frac{\left(-1\right)^{k-1}}{2\,k-1}\,\left(\frac
{\sin\left(x\right)-\cos\left(x\right)}
{\sin\left(x\right)+\cos\left(x\right)}\right)^{2\,k-1}\;.
\]
and the iteration formula
\[
x_{n+1}=x_n-\sum_{k=1}^{P}\,\frac{\left(-1\right)^{k-1}}{2\,k-1}\,\left(\frac
{\sin\left(x_n\right)-\cos\left(x_n\right)}
{\sin\left(x_n\right)+\cos\left(x_n\right)}\right)^{2\,k-1};\;\;n=0,1,\ldots
\]
\subsubsection{2-Term formulas}
According to Störmer [4], there are exactly four formulas with two terms:
\begin{align*}
&\frac{\pi}{4}=\arctan\frac{1}{2}+\arctan\frac{1}{3}
\\
&\frac{\pi}{4}=2\arctan\frac{1}{2}-\arctan\frac{1}{7}
\\
&\frac{\pi}{4}=2\arctan\frac{1}{3}+\arctan\frac{1}{7}
\\
&\frac{\pi}{4}=4\arctan\frac{1}{5}-\arctan\frac{1}{239}
\end{align*}
$\\ $
$\\ $
We start with
\[\frac{\pi}{4}
=\arctan\left(\frac{1}{2}\right)+\arctan\left(\frac{1}{3}\right)
\]
For $\arctan\left(\frac{1}{2}\right)$, that is, $q=2$, we get:
\[
T\left(x\right)=x-\sum_{k=1}^{P}\,\frac{\left(-1\right)^{k-1}}{2\,k-1}\,\left(\frac
{2\cdot\sin\left(x\right)-\cos\left(x\right)}
{2\cdot\cos\left(x\right)+\sin\left(x\right)}
\right)^{2\,k-1}
\]
and
\[
x_{n+1}=x_{n}-\sum_{k=1}^{P}\,\frac{\left(-1\right)^{k-1}}{2\,k-1}\,\left(\frac
{2\cdot\sin\left(x_{n}\right)-\cos\left(x_{n}\right)}
{2\cdot\cos\left(x_{n}\right)+\sin\left(x_{n-}\right)}
\right)^{2\,k-1};\;\;n=0,1,\ldots
\]
Evaluating $\arctan\left(\frac{1}{3}\right)$ yields
\[
T\left(x\right)=x-\sum_{k=1}^{P}\,\frac{\left(-1\right)^{k-1}}{2\,k-1}\,\left(\frac
{3\cdot\sin\left(x\right)-\cos\left(x\right)}
{3\cdot\cos\left(x\right)+\sin\left(x\right)}
\right)^{2\,k-1}
\]
\[
x_{n+1}=x_{n}-\sum_{k=1}^{P}\,\frac{\left(-1\right)^{k-1}}{2\,k-1}\,\left(\frac
{3\cdot\sin\left(x_{n}\right)-\cos\left(x_{n}\right)}
{3\cdot\cos\left(x_{n}\right)+\sin\left(x_{n}\right)}
\right)^{2\,k-1};\;\;n=0,1,\ldots
\]
$ \\ $
The next formulas are
\[\frac{\pi}{4}=2\arctan\left(\frac{1}{2}\right)-\arctan\left(\frac{1}{7}\right)\]
\[\frac{\pi}{4}=2\arctan\left(\frac{1}{3}\right)+\arctan\left(\frac{1}{7}\right)\]
We already have the fixed-point function and the iterative formula for $\arctan\left(\frac{1}{2}\right)$ and $\arctan\left(\frac{1}{3}\right)$, so we only need them for $\arctan\left(\frac{1}{7}\right)$. This gives
\[
T\left(x\right)=x-\sum_{k=1}^{P}\,\frac{\left(-1\right)^{k-1}}{2\,k-1}\,\left(\frac
{7\cdot\sin\left(x\right)-\cos\left(x\right)}
{7\cdot\cos\left(x\right)+\sin\left(x\right)}
\right)^{2\,k-1}
\]
\[
x_{n+1}=x_{n}-\sum_{k=1}^{P}\,\frac{\left(-1\right)^{k-1}}{2\,k-1}\,\left(\frac
{7\cdot\sin\left(x_{n}\right)-\cos\left(x_{n}\right)}
{7\cdot\cos\left(x_{n}\right)+\sin\left(x_{n}\right)}
\right)^{2\,k-1};\;\;n=0,1,\ldots
\]
$ \\ $
Now we come to Machin’s famous formula
\[\frac{\pi}{4}
=4\arctan\left(\frac{1}{5}\right)-\arctan\left(\frac{1}{239}\right)
\]
For $\arctan\left(\frac{1}{5}\right)$, we get
\[
T\left(x\right)=x-\sum_{k=1}^{P}\,\frac{\left(-1\right)^{k-1}}{2\,k-1}\,\left(\frac
{5\cdot\sin\left(x\right)-\cos\left(x\right)}
{5\cdot\cos\left(x\right)+\sin\left(x\right)}
\right)^{2\,k-1}
\]
\[
x_{n+1}=x_{n}-\sum_{k=1}^{P}\,\frac{\left(-1\right)^{k-1}}{2\,k-1}\,\left(\frac
{5\cdot\sin\left(x_{n}\right)-\cos\left(x_{n}\right)}
{5\cdot\cos\left(x_{n}\right)+\sin\left(x_{n}\right)}
\right)^{2\,k-1};\;\;n=0,1,\ldots
\]
Finally $\arctan\left(\frac{1}{239}\right)$ gives
\[
T\left(x\right)=x-\sum_{k=1}^{P}\,\frac{\left(-1\right)^{k-1}}{2\,k-1}\,\left(\frac
{239\cdot\sin\left(x\right)-\cos\left(x\right)}
{239\cdot\cos\left(x\right)+\sin\left(x\right)}
\right)^{2\,k-1}
\]
\[
x_{n+1}=x_{n}-\sum_{k=1}^{P}\,\frac{\left(-1\right)^{k-1}}{2\,k-1}\,\left(\frac
{239\cdot\sin\left(x_{n}\right)-\cos\left(x_{n}\right)}
{239\cdot\cos\left(x_{n}\right)+\sin\left(x_{n}\right)}
\right)^{2\,k-1};\;\;n=0,1,\ldots
\]
More Machin-like formulas with three and more terms are detailed described in \cite{JA}.
\subsection{\texorpdfstring{Computing one million digits of\;$\frac{\pi}{4}$}{}}
We want to compute one million digits of $\frac{\pi}{4}$. That means $\epsilon=10^{-1000000}$.  We use the formula for one term: $\frac{\pi}{4}=\arctan\left(1\right)$. In order to use the fixed point function efficiently for high-precision computation, a fast evaluation of $\sin\left(x\right)$ and $\cos\left(x\right)$ is required. There are many interesting publications on this topic \cite{LT},\cite{RB},\cite{HG},\cite{RBA},\cite{NB}. From the large number of proposed methods, we have selected the following three:
$\\ $
$\\ $
\underline{\bf{1.\;Selecting an initial value very close to $\frac{\pi}{4}$}}
$\\ $
We use as starting value
\[x_0=0.7853981633975\]
which is already accurate to 14 digits.
$ \\ $
$ \\ $
\underline{\bf{2.\;Increase in order of convergence}}
$\\ $
To ensure that the order of convergence takes full effect, we choose $P=2$. Next, we determine the fixed point function 
\[T\left(x\right)=x
-\frac{\sin\left(x\right)-\cos\left(x\right)}{\sin\left(x\right)+\cos\left(x\right)}
+\frac{1}{3}\left(\frac{\sin\left(x\right)-\cos\left(x\right)}{\sin\left(x\right)+\cos\left(x\right)}\right)^3\]
and the fixed point iteration:
\[x_{n+1}
=x_{n}-\frac
{\sin\left(x_{n}\right)-\cos\left(x_{n}\right)}
{\sin\left(x_{n}\right)+\cos\left(x_{n}\right)}
+\frac{1}{3}\left(
\frac
{\sin\left(x_{n}\right)-\cos\left(x_{n}\right)}
{\sin\left(x_{n}\right)+\cos\left(x_{n}\right)}\right)^{3}\,,
\]
which we use in the form
\[x_{n}
=x_{n-1}-\frac
{\sin\left(x_{n-1}\right)-\cos\left(x_{n-1}\right)}
{\sin\left(x_{n-1}\right)+\cos\left(x_{n-1}\right)}
+\frac{1}{3}\left(
\frac
{\sin\left(x_{n-1}\right)-\cos\left(x_{n-1}\right)}
{\sin\left(x_{n-1}\right)+\cos\left(x_{n-1}\right)}\right)^{3}\;,\;\;n=1,2,\ldots
\]
The index $n$ then runs in step with the number of iterations. $P=2$ gives order of convergence $2\cdot 2+1=5$. Then $x_n$ converges quintically to $\frac{\pi}{4}$, that is, each iteration approximately multiplies the number of correct digits by five. Thus we need eight iterations to obtain at least one million digits of $\frac{\pi}{4}$. 
$ \\ $
$ \\ $
\underline{\bf{3.\;Increase in the number of Digits}}
$\\ $
The fixed point iteration (like any other) is self-correcting. Each iteration must not be performed with the desired number of correct digits of $\frac{\pi}{4}$'s final result.
$\\ $
$\\ $
For the calculations, we use the computer algebra systems Maple. The \textbf{Digits} environment variable controls the number of digits that Maple uses when making calculations with software floating-point numbers. In step 1 we then need an accuracy of $14\cdot 5=70$ \textbf{Digits}. In step 2 the needed accuracy is $14\cdot 5^2=350$ \textbf{Digits} and so on until step 7 with accuracy of $14\cdot 5^7=1093750$ \textbf{Digits}. We have now reached the desired accuracy of at least one million digits. Therefore, in the final step 8, we no longer need to multiply the number of \textbf{Digits} by 5.
$ \\ $
Before starting the computation we take a look at the involved functions:
\[y=x\;;\quad f\left(x\right)=\frac{\sin\left(x\right)-\cos\left(x\right)}{\sin\left(x\right)+\cos\left(x\right)}\;;\quad
T\left(x\right)=x
-\frac{\sin\left(x\right)-\cos\left(x\right)}{\sin\left(x\right)+\cos\left(x\right)}
+\frac{1}{3}\left(\frac{\sin\left(x\right)-\cos\left(x\right)}{\sin\left(x\right)+\cos\left(x\right)}\right)^3\,.\]
\begin{center}
\fbox{\includegraphics[width=0.8\linewidth]{"fp_pi4_P2"}}
\centerline{\textbf{\underline{figure 1}}}
\end{center}
\centerline{ }
In the following $\frac{\pi}{4}$ calculation, we have implemented the three optimisation methods described above. In addition, we output the number of \textbf{Digits} with which Maple calculates at each step. We start with the initial value
$x_0=0.7853981633975$ and compute eight steps
\[\underline{step\;\;1}\]
\[Digits=14\cdot 5=70\]
\[x_{{1}}=x_0-{\frac
{\sin\left(x_{{0}}\right)-\cos\left(x_{{0}}\right)}
{\sin\left(x_{{0}}\right)+\cos\left(x_{{0}}\right)}}
+\frac{1}{3}\left(
{\frac{\sin\left(x_{{0}}\right)-\cos\left(x_{{0}}\right)}
{\sin\left(x_{{0}}\right)+\cos\left(x_{{0}}\right)}}\right)^{3}\]
\[x_{{1}}=0.7853981633974483096156608458198757210493\]
\[\left|x_{{1}}-x_{{0}}\right|={5.169038433915418012427895070765015622354\times10^{-14}}\]

\[\underline{step\;\;2}\]
\[Digits=14\cdot 5^2=350\]
\[x_{{2}}=x_1-{\frac
{\sin\left(x_{{1}}\right)-\cos\left(x_{{1}}\right)}
{\sin\left(x_{{1}}\right)+\cos\left(x_{{1}}\right)}}
+\frac{1}{3}\left(
{\frac
{\sin\left(x_{{1}}\right)-\cos\left(x_{{1}}\right)}
{\sin\left(x_{{1}}\right)+\cos\left(x_{{1}}\right)}}\right)^{3}\]
\[x_{{2}}=0.7853981633974483096156608458198757210493\]
\[\left|x_{{2}}-x_{{1}}\right|={7.369842844775034299129366447073300446298\times10^{-68}}\]
\[\underline{step\;\;3}\]
\[Digits=14\cdot 5^3=1750\]
\[x_{{3}}=x_2-{\frac
{\sin\left(x_{{2}}\right)-\cos\left(x_{{2}}\right)}
{\sin\left(x_{{2}}\right)+\cos\left(x_{{2}}\right)}}
+\frac{1}{3}\left(
{\frac
{\sin\left(x_{{2}}\right)-\cos\left(x_{{2}}\right)}
{\sin\left(x_{{2}}\right)+\cos\left(x_{{2}}\right)}}\right)^{3}\]
\[x_{{3}}=0.7853981633974483096156608458198757210493\]
\[\left|x_{{3}}-x_{{2}}\right|={4.348316332835180268524099971673673627949\times10^{-337}}\]
\[\underline{step\;\;4}\]
\[Digits=14\cdot 5^4=8750\]
\[x_{{4}}=x_3-{\frac
{\sin\left(x_{{3}}\right)-\cos\left(x_{{3}}\right)}
{\sin\left(x_{{3}}\right)+\cos\left(x_{{3}}\right)}}
+\frac{1}{3}\left(
{\frac
{\sin\left(x_{{3}}\right)-\cos\left(x_{{3}}\right)}
{\sin\left(x_{{3}}\right)+\cos\left(x_{{3}}\right)}}\right)^{3}\]
\[x_{{4}}=0.7853981633974483096156608458198757210493\]
\[\left|x_{{4}}-x_{{3}}\right|={3.109106863680802672512752062886498044627\times10^{-1683}}\]
\[\underline{step\;\;5}\]
\[Digits=14\cdot 5^5=43750\]
\[x_{{5}}=x_4-{\frac
{\sin\left(x_{{4}}\right)-\cos\left(x_{{4}}\right)}
{\sin\left(x_{{4}}\right)+\cos\left(x_{{4}}\right)}}
+\frac{1}{3}\left(
{\frac
{\sin\left(x_{{4}}\right)-\cos\left(x_{{4}}\right)}
{\sin\left(x_{{4}}\right)+\cos\left(x_{{4}}\right)}}\right)^{3}\]
\[x_{{5}}=0.7853981633974483096156608458198757210493\]
\[\left|x_{{5}}-x_{{4}}\right|={5.810429595425326132376821459074554489163\times10^{-8414}}\]
\[\underline{step\;\;6}\]
\[Digits=14\cdot 5^6=218750\]
\[x_{{6}}=x_5-{\frac
{\sin\left(x_{{5}}\right)-\cos\left(x_{{5}}\right)}
{\sin\left(x_{{5}}\right)+\cos\left(x_{{5}}\right)}}
+\frac{1}{3}\left(
{\frac
{\sin\left(x_{{5}}\right)-\cos\left(x_{{5}}\right)}
{\sin\left(x_{{5}}\right)+\cos\left(x_{{5}}\right)}}\right)^{3}\]
\[x_{{6}}=0.7853981633974483096156608458198757210493\]
\[\left|x_{{6}}-x_{{5}}\right|={1.324558707074286945224127205648362976460\times10^{-42067}}\]
\[\underline{step\;\;7}\]
\[Digits=14\cdot 5^7=1093750\]
\[x_{{7}}=x_6-{\frac
{\sin\left(x_{{6}}\right)-\cos\left(x_{{6}}\right)}
{\sin\left(x_{{6}}\right)+\cos\left(x_{{6}}\right)}}
+\frac{1}{3}\left(
{\frac
{\sin\left(x_{{6}}\right)-\cos\left(x_{{6}}\right)}
{\sin\left(x_{{6}}\right)+\cos\left(x_{{6}}\right)}}\right)^{3}\]
\[x_{{7}}=0.7853981633974483096156608458198757210493\]
\[\left|x_{{7}}-x_{{6}}\right|={8.154288164686902937605883297903247653927\times10^{-210336}}\]
\[\underline{step\;\;8}\]
\[Digits=1093750\]
\[x_{{8}}=x_7-{\frac
{\sin\left(x_{{7}}\right)-\cos\left(x_{{7}}\right)}
{\sin\left(x_{{7}}\right)+\cos\left(x_{{7}}\right)}}
+\frac{1}{3}\left(
{\frac
{\sin\left(x_{{7}}\right)-\cos\left(x_{{7}}\right)}
{\sin\left(x_{{7}}\right)+\cos\left(x_{{7}}\right)}}\right)^{3}\]
\[x_{{8}}=0.7853981633974483096156608458198757210493\]
\[\left|x_{{8}}-x_{{7}}\right|={7.210415146280134406533351725774712443819\times10^{-1051677}}\]
$\\ $
With eight steps, we obtained well over a million digits of $\frac{\pi}{4}$. The negative exponents of the differences show the convergence order $5$. They approximately increase by a factor of 5.
$\\ $
$\\ $
We performed all computational results by using a home-made PC with the following hardware configuration: Motherboard ASUS PRIME A320M-K with CPU AMD Ryzen 5 5600G 6 CORE 3.90-4.40 GHz and 32 GB RAM. The used software was MAPLE 2025.2 by  Maplesoft, Waterloo Maple Inc. This software was also used to assist with formatting the mathematical terms into LaTeX.

\newpage
\selectlanguage{ngerman}
\begin{center}
{\LARGE\bf{Ein iteratives Verfahren\\}}
\vskip 0.25cm
{\LARGE\bf{zur Berechnung von arctan(x) mit}}
\vskip 0.25cm
{\LARGE\bf{beliebiger ungerader Konvergenzordnung }}
\[\]
\[\]
\[\]
\large
{Alois Schiessl}
\footnote[1]{{University of Regensburg}\,\;E-Mail: \texttt{aloisschiessl@posteo.de}}
\selectlanguage{ngerman}
\vskip 0.5cm
\centerline{\today}
\end{center}
\[\]
\centerline{ }
\begin{abstract}
\selectlanguage{ngerman}
In dieser Abhandlung stellen wir ein Fixpunktverfahren zur Berechnung des Arcustangens mit beliebiger ungerader Konvergenzordnung vor.
$ \\ $
$ \\ $
Es sei $t\in\mathbb{R}^{+}$ und $P\in \mathbb{N}$. Wir definieren die Fixpunktiteration
\[T\left(x\right)=x-\sum_{k=1}^{P}\,\frac{\left(-1\right)^{k-1}}{2\,k-1}
\left(\frac{\sin\!\left(x\right)-t\cos\!\left(x\right)}{\cos\!\left(x\right)+t\sin\!\left(x\right)}\right)
^{2\,k-1}.\]
Für jeden Startwert $x_0$ hinreichend nahe bei $\arctan\left(t\right)$ konvergiert die Folge
\begin{align*}
x_{n+1}&=T\left(x_{n}\right)\;;\;n=0\,,1\,,\ldots
\end{align*}
gegen $\arctan\left(t\right)$ mit Konvergenzordnung genau $\left(2\,P+1\right)$. Bei Anwendung des Verfahrens auf Formeln vom Machin'schen Typ ergeben sich effiziente Berechnungs-Methoden für $\frac{\pi}{4}$.
\end{abstract}
\setcounter{section}{0}
\section{Einleitung und Theorem}
We beginnen mit der Potenzreihe der Arcustangens Funktion \cite{AS}
\[
\arctan\left(u\right)
=u-\frac{u^3}{3}+\frac{u^5}{5}-\frac{u^7}{7}+-\dots
=\sum_{k=1}^{\infty}\,\frac{\left(-1\right)^{k-1}}{2\,k-1}\,u^{\,2\,k-1}\,.
\]
Sie ist absolut konvergent für $\left|u\right|< 1$. Wir verwenden nur die ersten $P$ Summanden: 
\[
u-\frac{u^3}{3}+\frac{u^5}{5}-\frac{u^7}{7}+-\ldots+\frac{\left(-1\right)^{P-1}}{2\,P-1}\,u^{2\,P-1}
=\sum_{k=1}^{P}\frac{\left(-1\right)^{k-1}}{2\,k-1}\,u^{2\,k-1}\,.
\]
Jetzt substituieren wir
\[u=\frac{\sin\!\left(x\right)-t\cos\!\left(x\right)}{\cos\!\left(x\right)+t\sin\!\left(x\right)}\]
und erhalten die endliche Summe
\[
\sum_{k=1}^{P}\,\frac{\left(-1\right)^{k-1}}{2\,k-1}
\left(\frac
{\sin\!\left(x\right)-t\cos\!\left(x\right)}
{\cos\!\left(x\right)+t\sin\!\left(x\right)}
\right)^{2\,k-1}\,.
\]
Unter Verwendung dieser Summe definieren wir eine Fixpunktfunktion, die die Berechnung von $\arctan\left(t\right)$ mit beliebiger ungerader Konvergenzordnung ermöglicht. Da $\arctan\left(-t\right)$=$-\arctan\left(t\right)$ gilt, genügt es, nur positive Werte zu betrachten.
$\\ $
$\\ $
Nun formulieren wir unsere Aussage:
\begin{theorem}
Es sei $t\in\mathbb{R}^{+}$ und $P\in \mathbb{N}.$ Wir definieren die Funktion:
\[T:\;\;[\,0\,,\frac{\pi}{2}]\;\;\rightarrow\mathbb{R}\;;\]
\[x\mapsto T\left(x\right)\;;\]
\[
T\left(x\right)=x-\sum_{k=1}^{P}\,\frac{\left(-1\right)^{k-1}}{2\,k-1}
\left(\frac{\sin\!\left(x\right)-t\cos\!\left(x\right)}{\cos\!\left(x\right)+t\sin\!\left(x\right)}\right)
^{2\,k-1}.\]
Für die Funktion $T\left(x\right)$ gelten die folgenden Aussagen
\begin{itemize}
\item[(a)]
$T\left(x\right)$ ist eine Fixpunktfunktion, das heißt es gilt: $T\bigl(\arctan\left(t\right)\bigr)=\arctan\left(t\right)$
\item[(b)]
Die Ableitungen im Fixpunkt $\arctan\left(t\right)$ sind wie folgt gegeben
\begin{align*}
T^{\left(k\right)}{\bigl(\arctan\left(t\right)\bigr)}=\left\{
\begin{aligned}
0\quad\qquad\qquad\qquad\quad 0\le &\,k\leq 2\, P
\\
\\
\left(-1\right)^{P}\left(2\cdot P\right)\,!
\qquad\qquad\quad k=\,&2\, P+1
\\ 
\end{aligned} 
\right.
\end{align*}
\item[(c)] Für jeden Startwert $x_0$ hinreichend nahe bei $\arctan\left(t\right)$ konvergiert die Folge
\begin{align*}
x_{n+1}&=T\left(x_{n}\right)\;;\;n=0\,,1\,,\ldots
\end{align*}
gegen $\arctan\left(t\right)$ mit Konvergenzordnung genau $\left(2\,P+1\right)$.
\end{itemize}
\end{theorem}
\section{Beweis Theorem}
In diesem Kapitel weisen wir die Aussagen zum Theorem nach.
\subsection{Beweis Fixpunkt}
Um die Schreibarbeit für
\[
T\left(x\right)=x-\sum_{k=1}^{P}\,\frac{\left(-1\right)^{k-1}}{2\,k-1}
\left(\frac{\sin\!\left(x\right)-t\cos\!\left(x\right)}{\cos\!\left(x\right)+t\sin\!\left(x\right)}\right)^{2\,k-1}
\]
zu vereinfachen definieren wir eine neue Funktion:
$\\ $
Es sei: $t\in\mathbb{R^+}$
\[f:\;\;[0,\frac{\pi}{2}]\;\;\rightarrow\mathbb{R}\;;\]
\[f\left(x\right)=\frac{\sin\!\left(x\right)-t\cos\!\left(x\right)}{\cos\!\left(x\right)+t\sin\!\left(x\right)}\,.\]
Die Fixpunktfunktion schreibt sich dann einfacher als
\[
T\left(x\right)=x-\sum_{k=1}^{P}\,\frac{\left(-1\right)^{k-1}}{2\,k-1}f\left(x\right)^{2\,k-1}
\]
Wir haben zu beweisen, dass im Fixpunkt $x=\arctan\left(t\right)$ gilt:
\[
T\bigl(\arctan\left(t\right)\bigr)=\arctan\left(t\right).
\]
Es genügt zu zeigen, dass folgende Aussage zutrifft:
\[
f\bigl(\arctan\left(t\right)\bigr)=\frac{\sin\,\bigl(\arctan\left(t\right)\bigr)-t\cos\,\bigl(\arctan\left(t\right)\bigr)}{\cos\,\bigl(\arctan\left(t\right)\bigr)+t\sin\,\bigl(\arctan\left(t\right)\bigr)}=0\,.
\]
Einem mathematischen Handbuch \cite{SO} entnehmen wir die folgenden Formeln:
\[\sin\,\bigl(\arctan\left(t\right)\bigr)=\frac{t}{\sqrt{1+t^2}}\;;\;t\in\mathbb{R}\;;\]
\[\cos\,\bigl(\arctan\left(t\right)\bigr)=\frac{1}{\sqrt{1+t^2}}\;;\;t\in\mathbb{R}\;.\]
Dies eingesetzt ergibt
\[
f\bigl(\arctan\left(t\right)\bigr)
=\frac{\frac{t}{\sqrt{1+t^2}}-t\frac{1}{\sqrt{1+t^2}}}{\frac{1}{\sqrt{1+t^2}}+t\frac{t}{\sqrt{1+t^2}}}
=\frac{\frac{t}{\sqrt{1+t^2}}-\frac{t}{\sqrt{1+t^2}}}{\frac{1}{\sqrt{1+t^2}}+\frac{t^2}{\sqrt{1+t^2}}}
=\frac{0}{\frac{1}{\sqrt{1+t^2}}+\frac{t^2}{\sqrt{1+t^2}}}=0\,.
\]
Hieraus folgt unmittelbar
\[
T\bigl(\arctan\left(t\right)\bigr)=\arctan\left(t\right)-\sum_{k=1}^{P}\,\frac{\left(-1\right)^{k-1}}{2\,k-1}
{\underbrace{f\bigl(\arctan\left(t\right)\bigr)}_{=0}}^{2\,k-1}=\arctan\left(t\right)\,.
\]
Damit ist die Fixpunkteigenschaft bewiesen.
\subsection{Beweis der Ableitungen}
Als nächstes wenden wir uns den Ableitungen zu. Dazu ist etwas mehr Arbeit nötig.
\subsubsection{Erste Ableitung}
Wir beginnen mit der ersten Ableitung. Die Fixpunktfunktion verwenden wir in der Kurz-Schreibweise
\[
T\left(x\right)=x-\sum_{k=1}^{P}\,\frac{\left(-1\right)^{k-1}}{2\,k-1}f\left(x\right)^{2\,k-1}
\]
mit
\[
f\left(x\right)=\frac{\sin\!\left(x\right)-t\cos\!\left(x\right)}{\cos\!\left(x\right)+t\sin\!\left(x\right)}\,.
\]
Wir differenzieren die Fixpunktfunktion und erhalten
\begin{align*}
T\,'\left(x\right)
&=\left(x-\sum_{k=1}^{P}\,\frac{\left(-1\right)^{k-1}}{2\,k-1}f\left(x\right)^{2\,k-1}\right)'
\\
&=1-\sum_{k=1}^{P}\left(-1\right)^{k-1}f\left(x\right)^{2\,k-2}f\,'\left(x\right)\,.
\\
&=1+\sum_{k=1}^{P}\left(-1\right)^{k}f\left(x\right)^{2\,k-2}f\,'\left(x\right)\,.
\end{align*}
Für das weitere benötigen wir die Ableitung von
\[
f\left(x\right)=\frac{\sin\!\left(x\right)-t\cos\!\left(x\right)}{\cos\!\left(x\right)+t\sin\!\left(x\right)}\,.
\]
Wir übertragen diese Aufgabe dem Computeralgebrasystem Maple und vertrauen auf das Ergebnis:
\[
f\,'\left(x\right)=1+\left(\frac{\sin\!\left(x\right)-t\cos\!\left(x\right)}{\cos\!\left(x\right)+t\sin\!\left(x\right)}\right)^2=1+f\left(x\right)^2.
\]
Dies eingesetzt ergibt dann
\begin{align*}
T\,'\left(x\right)
&=1+\sum_{k=1}^{P}\left(-1\right)^{k-1}f\left(x\right)^{2\,k-2}\left(1+f\left(x\right)^2\right)
\\
&=1+\sum_{k=1}^{P}\left(-1\right)^{k}f\left(x\right)^{2\,k-2}
+\sum_{k=1}^{P}\left(-1\right)^{k}f\left(x\right)^{2\,k}.
\end{align*}
Um die Angelegenheit übersichtlicher zu gestalten bezeichnen wir die erste Summe mit
\begin{align*}
s_1=1+\sum_{k=1}^{P}\left(-1\right)^{k}f\left(x\right)^{2\,k-2}
\end{align*}
und die zweite mit 
\[
s_2=\sum_{k=1}^{P}\left(-1\right)^{k}f\left(x\right)^{2\,k},.
\]
Die erste Ableitung schreibt sich dann ganz einfach als
\[
T\,'\left(x\right)=s_1+s_2\,.
\]
Als nächstes rechnen wir die Summe $s_1$ teilweise aus, indem wir den ersten Summanden ausrechnen und die Indizierung der restlichen Summe bei 2 beginnen.
\begin{align*}
s_1
&=1+\sum_{k=1}^{1}\left(-1\right)^{k}f\left(x\right)^{2\,k-2}
+\sum_{k=2}^{P}\left(-1\right)^{k}f\left(x\right)^{2\,k-2}
\\
&=1-1
+\sum_{k=2}^{P}\left(-1\right)^{k}f\left(x\right)^{2\,k-2}
\\
&=
\sum_{k=2}^{P}\left(-1\right)^{k}f\left(x\right)^{2\,k-2}.
\end{align*}
Dann nehmen wir eine Index-Verschiebung vor
\begin{align*}
s_1
&=\sum_{k=1}^{P-1}\left(-1\right)^{k+1}f\left(x\right)^{2\,k}
\\
=&-\sum_{k=1}^{P-1}\left(-1\right)^{k}f\left(x\right)^{2\,k}\,.
\end{align*}
Ähnlich verfahren wir mit der Summe $s_2$. Dies mal lassen wir die ersten $\left(P-1\right)$ Summanden stehen und rechnen den letzten Term aus:
\begin{align*}
s_2&=\sum_{k=1}^{P-1}\left(-1\right)^{k}f\left(x\right)^{2\,k}
+\sum_{k=P}^{P}\left(-1\right)^{k}f\left(x\right)^{2\,k}
\\
&=\sum_{k=1}^{P-1}\left(-1\right)^{k}f\left(x\right)^{2\,k}
+\left(-1\right)^{P}f\left(x\right)^{2\,P}.
\end{align*}
Bie jetzt haben wir
\begin{align*}
s_1
&=-\sum_{k=1}^{P-1}\left(-1\right)^{k}f\left(x\right)^{2\,k}
\\
s_2
&=\;\;\;\sum_{k=1}^{P-1}\left(-1\right)^{k}\bigl(f\left(x\right)\bigr)^{2\,k}
+\left(-1\right)^{P}f\left(x\right)^{2\,P}\,.
\end{align*}
Dieses eingesetzt ergibt dann:
\begin{align*}
T\,'\left(x\right)
=&\;\;\;s_1+s_2
\\
=&-\sum_{k=1}^{P-1}\left(-1\right)^{k}f\left(x\right)^{2\,k}
\\
&+\sum_{k=1}^{P-1}\left(-1\right)^{k}f\left(x\right)^{2\,k}
+\left(-1\right)^{P}f\left(x\right)^{2\,P}
\\
&=\left(-1\right)^{P}f\left(x\right)^{2\,P}\,.
\end{align*}
Die Summen verschwinden und es bleibt das einfache Resultat:
\[T\,'\left(x\right)
=\left(-1\right)^{P}f\left(x\right)^{2\,P}
=\left(-1\right)^{P}\left(\frac
{\sin\!\left(x\right)-t\cos\!\left(x\right)}
{\cos\!\left(x\right)+t\sin\!\left(x\right)}
\right)^{2\,P}\,.
\]
Als nächstes setzen wir den Fixpunkt $\arctan\left(t\right)$ ein und erhalten
\[
T\,'\bigl(\arctan\left(t\right)\bigr)=\left(-1\right)^{P}\left(\frac
{\sin\bigl(\arctan\left(t\right)\bigr)-t\cos\bigl(\arctan\left(t\right)\bigr)}
{\cos\bigl(\arctan\left(t\right)\bigr)+t\sin\bigl(\arctan\left(t\right)\bigr)}
\right)^2\,.
\]
Den Term in der Klammer kennen wir bereits vom Fixpunktbeweis her:
\[
\frac
{\sin\bigl(\arctan\left(t\right)\bigr)-t\cos\bigl(\arctan\left(t\right)\bigr)}
{\cos\bigl(\arctan\left(t\right)\bigr)+t\sin\bigl(\arctan\left(t\right)\bigr)}
=0\,.
\]
Damit erhalten wir das Ergebnis:
\[T\,'\bigl(\arctan\left(t\right)\bigr)=0\,.\]
\subsubsection{Höhere Ableitungen}
Wir haben noch zu beweisen, dass für die höheren Ableitungen im Fixpunkt folgendes gilt:
\begin{align*}
T^{\left(k\right)}{\bigl(\arctan\left(t\right)\bigr)}=\left\{
\begin{aligned}
0\;\quad\qquad\qquad\;\; 1\leq &\,k\le 2\,P
\\
\\
\left(-1\right)^{P}\left(2\,P\right)\,!
\qquad\quad k=&2\,P+1
\\ 
\end{aligned} 
\right.
\end{align*}
Dazu müssen wir $\left(2\,P+1\right)$ die Fixpunktfunktion differenzieren und dann $\arctan\left(t\right)$ einsetzen. Wir haben bereits die erste Ableitung  $T\,'\left(x\right)$ berechnet. Diese verwenden wir jetzt.
\begin{align*}
\left.T^{\left(2\,P+1\right)}\left(x\right)\right|_{x=\arctan\left(t\right)}
=\left.\left(T\,'\right)^{\left(2\,P\right)}\left(x\right)\right|_{x=\arctan\left(t\right)}
=\left.\left(\left(-1\right)^{P}f\left(x\right)^{2\,P}\right)^{\left(2\,P\right)}\right|_{x=\arctan\left(t\right)}\,.
\end{align*}
Auf der rechten Seite benötigen wir die $\left(2\,P\right)^{te}$ Ableitung von $f\left(x\right)^{2\,P}$. Zuerst berechnen wir einige spezielle Werte:
\[f\bigl(\arctan\left(t\right)\bigr)=0\,;\]
\[f'\bigl(\arctan\left(t\right)\bigr)=1+\Bigl.f\left(x\right)^2\Bigr|_{x=\arctan\left(t\right)}=1\,.\]
Dann entwickeln wir $f\left(x\right)$ in eine Taylorreihe um $\arctan\left(t\right)$ vom Grad $1$ mit Lagrange Restglied:
\[
f\left(x\right)
=f\bigl(\arctan\left(t\right)\bigr)
+f'\bigl(\arctan\left(t\right)\bigr)\bigl(x-\arctan\left(t\right)\bigr)
+\frac{f\,''\left(\xi\right)}{2}\bigl(x-\arctan\left(t\right)\bigr)^{2}\,.
\]
wobei $\xi$ zwischen $x$ und $\arctan\left(t\right)$ liegt.
$\\ $
Wir verwenden die bereits bekannten Ergebnisse $f\bigl(\arctan\left(t\right)\bigr)=0$ und $f'\bigl(\arctan\left(t\right)\bigr)=1$. Das ergibt dann
\[
f\left(x\right)
=\bigl(x-\arctan\left(t\right)\bigr)+\frac{f\,''\left(\xi\right)}{2}\bigl(x-\arctan\left(t\right)\bigr)^{2}\,.\]
Hier können wir $\bigl(x-\arctan\left(t\right)\bigr)$ ausklammern und erhalten:
\[
f\left(x\right)
=\bigl(x-\arctan\left(t\right)\bigr)\left(1+\frac{f\,''\left(\xi\right)}{2}\bigl(x-\arctan\left(t\right)\bigr)\right)\,.
\]
Als nächstes erheben wir die Gleichung in die $\left(2\,P\right)^{\,te}$ Potenz:
\[
f\left(x\right)^{2\,P}
=\bigl(x-\arctan\left(t\right)\bigr)^{2\,P}
\left(1+\frac{f{\,''}\left(\xi\right)}{2}\bigl(x-\arctan\left(t\right)\bigr)\right)^{2\,P}\,.
\]
Die höheren Ableitungen berechnen wir mit der allgemeinen Leibniz-Regel, die besagt:
\\
Seien  $u\left(x\right)$ und $v\left(x\right)$ zwei über demselben Intervall $n$-mal differenzierbare Funktionen, dann ist das Produkt $u\left(x\right)\cdot v\left(x\right)$ ebenfalls $n$-mal differenzierbar und es gilt: 
\[
\bigl(u\left(x\right)\cdot v\left(x\right)\bigr)^{\left(n\right)}
=\sum\limits_{k=0}^{n}\binom{n}{k}\,u\left(x\right)^{\left(k\right)}\cdot v\left(x\right)^{\left(n-k\right)}\,.
\]
Hierbei wollen wir unter $u\left(x\right)^{\left(0\right)}=u\left(x\right)$ und $v\left(x\right)^{\left(0\right)}=v\left(x\right)$ verstehen.
$\\ $
$\\ $
In unserem Fall haben wir:
\[
n=2\,P\;;
\]
\[
u\left(x\right)=\bigl(x-\arctan\left(t\right)\bigr)^{2\,P}\;;
\]
\[
v\left(x\right)=\left(1+\frac{f\,''\left(\xi\right)}{2}\bigl(x-\arctan\left(t\right)\bigr)\right)^{2\,P}\;;
\]
Wir erhalten dann
\[
f\left(x\right)^{2\,P}=
\underbrace{{\bigl(x-\arctan\left(t\right)\bigr)}^{2\,P}}_{u\left(x\right)}
\cdot
\underbrace{{{\left(1+\frac{f\,''{\left(\xi\right)}}{2}\bigl(x-\arctan\left(t\right)\bigr)\right)}^{2\,P}}}_{v\left(x\right)}\,.
\]
Hierauf wenden wir die Leibniz-Regel an:
\begin{align*}
\left(f\left(x\right)^{2\,P}\right)^{\left({2\,P}\right)}
&=\sum\limits_{k=0}^{{2\,P}}\binom{{2\,P}}{k}
\,\left(\bigl(x-\arctan\left(t\right)\bigr)^{2\,P}\right)^{\left(k\right)}
\\
&\times
\left(\left(1+\frac{f\,''\left(\xi\right)}{2}\bigl(x-\arctan\left(t\right)\bigr)\right)^{{2\,P}}\right)^{\left({2\,P-k}\right)}
\end{align*}
Wir sehen uns zuerst die Folge der Ableitungen an von
\[
u\left(x\right)^{\left(k\right)}
=\left(\bigl(x-\arctan\left(t\right)\bigr)^{{2\,P}}\right)^{\left(k\right)}\,,k=0,1,\ldots,{2\,P}\,.
\]
Zur Berechnung der Ableitungen von gibt es eine einfache Formel, die wir einer Formelsammlung  \cite{SO} entnehmen und auf unsere Bedürfnisse zuschneiden:
\begin{align*}
\Bigl(\bigl(x-\arctan\left(t\right)\bigr)^{2\,P}\Bigr)^{\left(k\right)}
=\frac{\left(2\,P\right)\,!}{\left(2\,P-k\right)\,!}\bigl(x-\arctan\left(t\right)\bigr)^{2\,P-k}\,,k=0,1,\ldots,{2\,P}.
\end{align*}
$\\ $
Insbesondere stellen wir fest, dass in der Folge
\[
u\left(x\right)^{\left(k\right)}=
\left(\bigl(x-\arctan\left(t\right)\bigr)^{{2\,P}}\right)^{\left(k\right)}
\]
für $k<2\,P$ stets der Faktor
$\bigl(x-\arctan\left(t\right)\bigr)$ auftritt und daher für $x=\arctan\left(t\right)$ die Terme verschwinden. Nur im Fall $k=\left(2\,P\right)$ ergibt sich ein von Null verschiedenes Resultat, nämlich das Produkt
\[
{\left.u{{\left(x\right)}^{\left(2\,P\right)}}\right|}_{x=\arctan\left(t\right)}=\left(2\,P\right)\,!\,.
\]
Von der Summe
\[
\sum\limits_{k=0}^{{2\,P}}\binom{{2\,P}}{k}
\,\left(\bigl(x-\arctan\left(t\right)\bigr)^{2\,P}\right)^{\left(k\right)}
\cdot
\left(\left(1+\frac{f\,''{\left(\xi\right)}}{2}\bigl(x-\arctan\left(t\right)\bigr)\right)^{{2\,P}}\right)^{\left({2\,P-k}\right)}
\]
bleibt im Endeffekt nur der Term $k=\left(2\,P\right)$ im Fixpunkt $\arctan\left(t\right)$ übrig. Das ergibt somit
\begin{align*}
{\left.{{\left(f{{\left(x\right)}^{2\,P}}\right)}^{\left(2\,P\right)}}\right|}_{x=\arctan\left(t\right)}
&=\underbrace{\sum\limits_{k={2\,P}}^{{2\,P}}\binom{{2\,P}}{{2\,P}}\left.\left(\bigl(x-\arctan\left(t\right)\bigr)^{2\,P}\right)^{\left(2\,P\right)}\right|_{x=\arctan\left(t\right)}}_{=\left(2\,P\right)\,!}
\\
&\times\underbrace{\left.\left(\left(1
+\frac{f\,''\left(\xi\right)}{2}\bigl(x-\arctan\left(t\right)\bigr)\right)^{2\,P}\right)\right|_{x=\arctan\left(t\right)}}_{=1}\,.
\end{align*}
Wir erhalten das einfache Ergebnis
\[
{\left.{{\left(f{{\left(x\right)}^{2\,P}}\right)}^{\left(2\,P\right)}}\right|}_{x=\arctan\left(t\right)}
=\left(2\,P\right)\,!\,.
\]
Unsere Aufgabe war die Berechnung von
\begin{align*}
\left.T\left(x\right)^{\left(2\,P+1\right)}\right|_{x=\arctan\left(t\right)}
=\left.T\,'\left(x\right)^{\left(2\,P\right)}\right|_{x=\arctan\left(t\right)}
=\left.\left(\left(-1\right)^{P}f\left(x\right)^{2\,P}\right)^{\left(2\,P\right)}\right|_{x=\arctan\left(t\right)}\,.
\end{align*}
Jetzt brauchen wir nur noch einzusetzen und erhalten das gewünschte Ergebnis.
\[
T\bigl(\arctan\left(t\right)\bigr)^{\left(2\,P+1\right)}
=\left(-1\right)^P\left(2\,P\right)\,!
\]
Damit ist der Beweis für die Ableitungen abgeschlossen.
\subsection{Beweis Konvergenz}
Wir führen den Beweis mit dem Fixpunktsatz von Banach und zwar in einer speziellen Form für stetig differenzierbare Abbildungen im ${{\mathbb{R}}^{1}}$.
$ \\ \\ $
\underline{Fixpunktsatz von Banach:}
\\
Es sei $U\subseteq R$ eine abgeschlossene Teilmenge. Weiter sei $F:U\rightarrow R$ eine Abbildung mit folgenden Eigenschaften
\[
F\left(U\right)\subseteq U \quad (Selbstabbildung)
\]
Es existiere ein $0<L<1$ mit der Eigenschaft
\[
L=\left|F\,'\left(x\right)\right|<1\,, \text{für\;\;alle} \; x\in U \quad (Kontraktionseigenschaft)\,.
\]
Dann gilt:
\begin{itemize}
\item[(a)]
Es gibt genau einen Fixpunkt ${{x}^{*}}\in U$ von $F$
\item[(b)]
Für jeden Startwert $x_0\in U$ konvergiert die Folge
$x_{n+1}=F\left(x_{{n}}\right)$
gegen $x^*$.
\end{itemize}
$ \\ $
\underline{Beweis der Konvergenz}
\\
Wir haben bereits nachgewiesen $T\bigl(\arctan\left(t\right)\bigr)=\arctan\left(t\right)$ und $T\,'\bigl(\arctan\left(t\right)\bigr)=0$. Dann existiert aus Stetigkeitsgründen um den Fixpunkt $\arctan\left(t\right)$ eine Umgebung
\[U_{\delta}=\left\{x\in\mathbb{R}:\left|x-\arctan\left(t\right)\right|\leq\delta\right\}\]
in der gilt: $T\left(U_{\delta}\right)\subseteq U_{\delta}$ und $\left|T\,'\left(x\right)\right|<L<1$ für alle $x\in U_{\delta}$. Somit sind die Voraussetzungen für den Fixpunktsatz erfüllt und hieraus folgt die Konvergenz.
\subsection{Konvergenzordnung}
Als letztes haben wir zu zeigen, dass die Konvergenzordnung genau $\left(2\,P+1\right)$ beträgt.
Wir greifen auf bereits bewährtes zurück: Taylorreihen. Diesmal entwickeln wir $T\left(x\right)$ in eine Taylorreihe um den Fixpunkt $\arctan\left(t\right)$ wie folgt:
\begin{align*}
T\left(x\right)&=T\bigl(\arctan\left(t\right)\bigr)
\\&
+T\,'\bigl(\arctan\left(t\right)\bigr)\bigl(x-\arctan\left(t\right)\bigr)
+\ldots+T^{\left(2\,P\right)}\bigl(\arctan\left(t\right)\bigr)
\frac{\bigl(x-\arctan\left(t\right)\bigr)^{2\,P}}{\left(2\,P\right)\,!}
\\
&+T^{\left(2\,P+1\right)}\left(\xi\right)
\frac{\bigl(x-\arctan\left(t\right)\bigr)^{2\,P+1}}{\left(2\,P+1\right)\,!}\;,
\end{align*}
mit einem $\xi$ zwischen $x$ und $\arctan\left(t\right)$.
$\\ $
Wir erinnern uns an die Fixpunkteigenschaft $T\bigl(\arctan\left(t\right)\bigr)=\arctan\left(t\right)$ und unsere Ableitungsergebnisse:
\begin{align*}
T^{\left(k\right)}{\bigl(\arctan\left(t\right)\bigr)}=\left\{
\begin{aligned}
0\quad\qquad\qquad\qquad\quad 0\le &\,k\leq 2\,P
\\
\\
\left(-1\right)^{P}\left(2\cdot P\right)\,!
\qquad\qquad\quad k=&\,2\, P+1
\\ 
\end{aligned} 
\right.
\end{align*}
Die Taylorreihe vereinfacht sich damit zu
\begin{align*}
T\left(x\right)
&=\underbrace{T\bigl(\arctan\left(t\right)\bigr)}_{=\arctan\left(t\right)}
\\
&+\underbrace{T\,'\bigl(\arctan\left(t\right)\bigr)\bigl(x-\arctan\left(t\right)\bigr)
+\ldots
+T^{\left(2\,P\right)}\bigl(\arctan\left(t\right)\bigr)
\frac{\bigl(x-\arctan\left(t\right)\bigr)^{2\,P}}{\left(2\,P\right)\,!}}_{=0}
\\
&+T^{\left(2\,P+1\right)}\left(\xi\right)
\frac{\bigl(x-\arctan\left(t\right)\bigr)^{2\,P+1}}{\left(2\,P+1\right)\,!}
\\
&=\arctan\left(t\right)+T\left(\xi\right)^{\left(2\,P+1\right)}
\frac{\bigl(x-\arctan\left(t\right)\bigr)^{2\,P+1}}{\left(2\,P+1\right)\,!}\,.
\end{align*}
Nun setzen wir $x=x_{n}$ hinreichend nahe bei $\arctan\left(t\right)$ ein und erhalten
\begin{align*}
\underbrace{T\left(x_{n}\right)}_{x_{n+1}}=
\arctan\left(t\right)+T^{\left(2\,P+1\right)}\left(\xi_{n}\right)
\frac{\left(x_{n}-\arctan\left(t\right)\right)^{2\,P+1}}{\left(2\,P+1\right)\,!}
\\
x_{n+1}=\arctan\left(t\right)
+T^{\left(2\,P+1\right)}\left(\xi_{n}\right)
\frac{\left(x_{n}-\arctan\left(t\right)\right)^{2\,P+1}}{\left(2\,P+1\right)\,!}\,.
\end{align*}
Wir bringen $\arctan\left(t\right)$ auf die linke Seite und dividieren durch
$\left(x_{n}-\arctan\left(t\right)\right)^{2\,P+1}$. Das ergibt dann
\[
\frac{x_{n+1}-\arctan\left(t\right)}{\left(x_{n}-\arctan\left(t\right)\right)^{2\,P+1}}=
T^{\left(2\,P+1\right)}\left(\xi_{n}\right)
\frac{1}{\left(2\,P+1\right)\,!}\,.
\]
Mit zunehmendem $n$ strebt $x_{n}\rightarrow \arctan\left(t\right)$. Da $\xi_{n}$ zwischen $x_{n}$ und $\arctan\left(t\right)$ eingesperrt ist, muss $\xi_{n}$ notgedrungen ebenfalls gegen $\arctan\left(t\right)$ streben. Wir erhalten also
\begin{align*}
\underset{n\to \infty}{\mathop{\lim }}\; \frac{x_{n+1}-\arctan\left(t\right)}{\left(x_{n}-\arctan\left(t\right)\right)^{2\,P+1}}
=\underset{{n}\to \infty}{\mathop{\lim }}\;
T^{\left(2\,P+1\right)}\left(\xi_{n}\right)
\frac{1}{\left(2\,P+1\right)\,!}
=T^{\left(2\,P+1\right)}\bigl(\arctan\left(t\right)\bigr)
\frac{1}{\left(2\,P+1\right)\,!}\,.
\end{align*}
Nun haben wir bereits berechnet
\[
T^{\left({2\,P+1}\right)}{\bigl(\arctan\left(t\right)\bigr)}
=\left(-1\right)^{P}\left(2\,P\right)\,!\,.
\]
Dies eingesetzt ergibt dann
\begin{align*}
\underset{n\to\infty}{\mathop{\lim }}\;\;
\frac{x_{n+1}-\arctan\left(t\right)}{\left(x_{n}-\arctan\left(t\right)\right)^{2\,P+1}}
=\left(-1\right)^{P}\frac{\left(2\,P\right)\,!}{\left(2\,P+1\right)\,!}
=\left(-1\right)^{P}\frac{\left(2\,P\right)\,!}{\left(2\,P\right)\,!\left(2\,P+1\right)}
=\frac{\left(-1\right)^{P}}{2\,P+1}\,.
\end{align*}
Hieraus folgt unmittelbar Konvergenzordnung mindestens $\left(2\,P+1\right)$.
$ \\ $
$ \\ $
Da $T{\bigl(\arctan\left(t\right)\bigr)}^{\left({2\,P+1}\right)}\neq 0$ haben wir genau Konvergenzordnung $\left(2\,P+1\right)$. Der auf der rechten Seite stehende konstante Term
\[
\frac{\left(-1\right)^{P}}{2\,P+1}
\]
wird als asymtotische Konvergenzrate bezeichnet.
\section{Praktische Anwendung}
Nun kommen wir zur Anwendung der Fixpunktfunktion. Es seien $t\in\mathbb{R}^{+}$ und $P\in\mathbb{N}$ vorgegeben. Wir erhalten dann die Konvergenzordnung $\left(2\,P+1\right)$. Für die Fixpunktfunktion ergibt sich
\[T
\left(x\right)=x-\sum_{k=1}^{P}
\frac{\left(-1\right)^{k-1}}{2\,k-1}
\left(\frac
{\sin\!\left(x\right)-t\cos\!\left(x\right)}
{\cos\!\left(x\right)+t\sin\!\left(x\right)}
\right)^{2\,k-1}
\]
Mit einem geeigneten Startwert $x_0$ berechnen wir die Folge:
\[x_{1}=x_{0}-\sum\limits_{k=1}^{P}
\frac{\left(-1\right)^{k-1}}{2\,k-1}
\left(\frac
{\sin\!\left(x_0\right)-t\cos\!\left(x_0\right)}
{\cos\!\left(x_0\right)+t\sin\!\left(x_0\right)}
\right)^{2\,k-1}\,;
\]

\[x_{2}=x_{1}-\sum\limits_{k=1}^{P}
\frac{\left(-1\right)^{k-1}}{2\,k-1}
\left(\frac
{\sin\!\left(x_1\right)-t\cos\!\left(x_1\right)}
{\cos\!\left(x_1\right)+t\sin\!\left(x_1\right)}
\right)^{2\,k-1}\,;
\]

\[x_{3}=x_{2}-\sum\limits_{k=1}^{P}
\frac{\left(-1\right)^{k-1}}{2\,k-1}
\left(\frac
{\sin\!\left(x_2\right)-t\cos\!\left(x_2\right)}
{\cos\!\left(x_2\right)+t\sin\!\left(x_2\right)}
\right)^{2\,k-1}\,;
\]
\[\vdots\]
\[x_{n+1}=x_{n}-\sum\limits_{k=1}^{P}
\frac{\left(-1\right)^{k-1}}{2\,k-1}
\left(\frac
{\sin\!\left(x_n\right)-t\cos\!\left(x_n\right)}
{\cos\!\left(x_n\right)+t\sin\!\left(x_n\right)}
\right)^{2\,k-1}\,;
\]
$ \\ $
Mit zunehmender Anzahl der Iterationsschritte strebt $x_{n+1}\rightarrow\arctan\left(t\right)$ mit Konvergenz-Ordnung $\left(2\,P+1\right)$. Bei beliebig langer Rechenzeit und beliebig viel Speicherplatz können wir damit $\arctan\left(t\right)$ beliebig genau berechnen. Leider steht uns nichts dergleichen zur Verfügung. Wir müssen uns mit einer endlichen Anzahl von Iterationen und einem endlichen Wert für $\arctan\left(t\right)$ begnügen.
$ \\ $
$ \\ $
Dazu geben wir $\epsilon>0$ vor und führen die Iteration solange aus bis erstmals gilt:
\[\left|x_{n+1}-x_{n}\right|<\epsilon\,.\]
\subsection{Iterationsformeln für Reihen vom Machin'schen Typ}
Nun kommen wir zur Berechnung von $\frac{\pi}{4}$ mit Formeln vom Machin' schen Typ. Das sind Reihen mit der Darstellung
\[
\frac{\pi}{4}=\sum_{k=1}^{n}a_n\arctan\frac{1}{b_n}
\]
wobei $a_n\in\mathbb Z$ und $b_n\in\mathbb N$.
$ \\ $
$ \\ $
Hier besteht jetzt das Argument $\frac{1}{b_n}$ aus einem Stammbruch. Der Vorteil liegt auf der Hand: je kleiner das Argument ist, umso weniger Summanden werden benötigt um eine vorgegebene Genauigkeit zu erreichen. Da das Argument jetzt ein Stammbruch $t=\frac{1}{q}$ mit $q\in\mathbb{N}$ ist, lässt sich die Fixpunktfunktion schreiben als
\[
T\left(x\right)
=x-\sum_{k=1}^{P}\,\frac{\left(-1\right)^{k-1}}{2\,k-1}\left(\frac
{\sin\left(x\right)-\frac{1}{q}\cdot\cos\left(x\right)}
{\cos\left(x\right)+\frac{1}{q}\cdot\sin\left(x\right)}
\right)^{2\,k-1}.
\]
Das lässt sich umformen zur einfachen Darstellung
\[
T\left(x\right)
=x-\sum_{k=1}^{P}\,\frac{\left(-1\right)^{k-1}}{2\,k-1}\,\left(\frac
{q\cdot\sin\left(x\right)-\cos\left(x\right)}
{q\cdot\cos\left(x\right)+\sin\left(x\right)}
\right)^{2\,k-1}.
\]
Davon leiten wir die dazugehörige Iterationsformel ab:
\[
x_{n+1}=x_{n}-\sum_{k=1}^{P}\,\frac{\left(-1\right)^{k-1}}{2\,k-1}\,\left(\frac
{q\cdot\sin\left(x_{n}\right)+\cos\left(x_{n}\right)}
{q\cdot\cos\left(x_{n}\right)+\sin\left(x_{n}\right)}
\right)^{2\,k-1};\;\;n=0,1,\ldots
\]
Diese Iterationsformel gilt für $\arctan\left(\frac{1}{q}\right)$ mit $q\in\mathbb{N}$.
\subsubsection{1-Term Formel}
Es gibt eine einzige Formel mit einem Term
\[
\frac{\pi}{4}=\arctan\left(1\right)
\]
Es ist also $q=1$ und wir erhalten
\[
T\left(x\right)=x-\sum_{k=1}^{P}\,\frac{\left(-1\right)^{k-1}}{2\,k-1}\,\left(\frac
{\sin\left(x\right)-\cos\left(x\right)}
{\sin\left(x\right)+\cos\left(x\right)}\right)^{2\,k-1}\;.
\]
\[
x_{n+1}=x_n+\sum_{k=1}^{P}\,\frac{\left(-1\right)^{k-1}}{2\,k-1}\,\left(\frac
{\sin\left(x_n\right)-\cos\left(x_n\right)}
{\sin\left(x_n\right)+\cos\left(x_n\right)}
\right)^{2\,k-1};\;\;n=0,1,\ldots
\]
\subsubsection{2-Term Formel}
Formeln mit zwei Termen gibt es nach Störmer \cite{CS} genau vier:
\begin{align*}
&\frac{\pi}{4}=\arctan\frac{1}{2}+\arctan\frac{1}{3}\;;
\\
&\frac{\pi}{4}=2\arctan\frac{1}{2}-\arctan\frac{1}{7}\;;
\\
&\frac{\pi}{4}=2\arctan\frac{1}{3}+\arctan\frac{1}{7}\;;
\\
&\frac{\pi}{4}=4\arctan\frac{1}{5}-\arctan\frac{1}{239}\;.
\end{align*}
Die dritte Formel ist keine unabhängige Formel. Sie lässt sich aus den ersten beiden herleiten. Die Formeln werden verschiedenen Entdeckern zugeschrieben, so Euler, Hermann, Hutton, u.a. Die letzte ist die populärste und wurde erstmals von Machin veröffentlicht. Jetzt brauchen wir nur noch die entsprechenden Nenner in obige Formel einsetzen.
$\\ $
$\\ $
Wir beginnen mit der Formel von Euler
\[\frac{\pi}{4}
=\arctan\left(\frac{1}{2}\right)+\arctan\left(\frac{1}{3}\right)\;.
\]
Wir erhalten für $\arctan\left(\frac{1}{2}\right)$, das heißt $q=2$:
\[
T\left(x\right)=x-\sum_{k=1}^{P}\,\frac{\left(-1\right)^{k-1}}{2\,k-1}\,\left(\frac
{2\cdot\sin\left(x\right)-\cos\left(x\right)}
{2\cdot\cos\left(x\right)+\sin\left(x\right)}
\right)^{2\,k-1}
\]
und die Iterationsformel
\[
x_{n+1}=x_{n}-\sum_{k=1}^{P}\,\frac{\left(-1\right)^{k-1}}{2\,k-1}\,\left(\frac
{2\cdot\sin\left(x_{n}\right)-\cos\left(x_{n}\right)}
{2\cdot\cos\left(x_{n}\right)+\sin\left(x_{n}\right)}
\right)^{2\,k-1};\;\;n=0,1,\ldots
\]
Für $\arctan\left(\frac{1}{3}\right)$ ergibt sich
\[
T\left(x\right)=x-\sum_{k=1}^{P}\,\frac{\left(-1\right)^{k-1}}{2\,k-1}\,\left(\frac
{3\cdot\sin\left(x\right)-\cos\left(x\right)}
{3\cdot\cos\left(x\right)+\sin\left(x\right)}
\right)^{2\,k-1}\;;
\]
\[
x_{n+1}=x_{n}-\sum_{k=1}^{P}\,\frac{\left(-1\right)^{k-1}}{2\,k-1}\,\left(\frac
{3\cdot\sin\left(x_{n}\right)-\cos\left(x_{n}\right)}
{3\cdot\cos\left(x_{n}\right)+\sin\left(x_{n}\right)}
\right)^{2\,k-1};\;\;n=0,1,\ldots
\]
$ \\ $
Die nächste Formeln sind
\[\frac{\pi}{4}=2\arctan\left(\frac{1}{2}\right)-\arctan\left(\frac{1}{7}\right)\]
\[\frac{\pi}{4}=2\arctan\left(\frac{1}{3}\right)+\arctan\left(\frac{1}{7}\right)\]
Fixpunktfunktion und Iterationsformel für $\arctan\left(\frac{1}{2}\right)$ und $\arctan\left(\frac{1}{3}\right)$ haben wir bereits, so dass wir sie nur noch für $\arctan\left(\frac{1}{7}\right)$ benötigen. Das ergibt
\[
T\left(x\right)=x-\sum_{k=1}^{P}\,\frac{\left(-1\right)^{k-1}}{2\,k-1}\,\left(\frac
{7\cdot\sin\left(x\right)-\cos\left(x\right)}
{7\cdot\cos\left(x\right)+\sin\left(x\right)}
\right)^{2\,k-1}\;;
\]
\[
x_{n+1}=x_{n}-\sum_{k=1}^{P}\,\frac{\left(-1\right)^{k-1}}{2\,k-1}\,\left(\frac
{7\cdot\sin\left(x_{n}\right)-\cos\left(x_{n}\right)}
{7\cdot\cos\left(x_{n}\right)+\sin\left(x_{n}\right)}
\right)^{2\,k-1};\;\;n=0,1,\ldots
\]
$ \\ $
Nun kommen wir zur berühmten Formel von Machin
\[\frac{\pi}{4}
=4\arctan\left(\frac{1}{5}\right)-\arctan\left(\frac{1}{239}\right)
\]
Zwischenzeitlich haben wir reichlich Übung, so dass es kein Problem mehr ist, Fixpunktfunktion und Iterationsformel zu erstellen. Wir erhalten für $\arctan\left(\frac{1}{5}\right)$
\[
T\left(x\right)=x-\sum_{k=1}^{P}\,\frac{\left(-1\right)^{k-1}}{2\,k-1}\,\left(\frac
{5\cdot\sin\left(x\right)-\cos\left(x\right)}
{5\cdot\cos\left(x\right)+\sin\left(x\right)}
\right)^{2\,k-1}\;;
\]
\[
x_{n+1}=x_{n}-\sum_{k=1}^{P}\,\frac{\left(-1\right)^{k-1}}{2\,k-1}\,\left(\frac
{5\cdot\sin\left(x_{n}\right)-\cos\left(x_{n}\right)}
{5\cdot\cos\left(x_{n}\right)+\sin\left(x_{n}\right)}
\right)^{2\,k-1};\;\;n=0,1,\ldots
\]
und für $\arctan\left(\frac{1}{239}\right)$
\[
T\left(x\right)=x-\sum_{k=1}^{P}\,\frac{\left(-1\right)^{k-1}}{2\,k-1}\,\left(\frac
{239\cdot\sin\left(x\right)-\cos\left(x\right)}
{239\cdot\cos\left(x\right)+\sin\left(x\right)}
\right)^{2\,k-1}\;;
\]
\[
x_{n+1}=x_{n}-\sum_{k=1}^{P}\,\frac{\left(-1\right)^{k-1}}{2\,k-1}\,\left(\frac
{239\cdot\sin\left(x_{n}\right)-\cos\left(x_{n}\right)}
{239\cdot\cos\left(x_{n}\right)+\sin\left(x_{n}\right)}
\right)^{2\,k-1};\;\;n=0,1,\ldots
\]
Für Formeln mit drei und mehr Summanden verweisen wir auf die die sehr ausführlichen Darstellungen von Jörg Arndt \cite{JA}. Hier werden Zerlegungen mit bis zu 27 Summanden angegeben.
\subsection{\texorpdfstring{Berechnung einer Million Stellen von\;$\frac{\pi}{4}$}{}}
Nun kommen wir zu einer praktischen Berechnung. Wir wollen $\frac{\pi}{4}$ auf mindestens eine Million Dezimalstellen genau berechnen. Das bedeutet $\epsilon=10^{-1000000}$. Um die Fixpunktfunktion zur Berechnung von $\frac{\pi}{4}$ verwenden zu können, ist eine effiziente Berechnung von $\sin\left(x\right)$ und $\cos\left(x\right)$ erforderlich. Es gibt jede Menge Veröffentlichungen zu diesem Thema \cite{LT},\cite{RB},\cite{HG},\cite{RBA},\cite{NB}. Aus der Vielzahl der vorgeschlagenen Verfahren haben wir die folgenden drei ausgewählt:
$\\ $
$\\ $
\underline{\bf{1.\;Verwendung eines Startwertes, der schon sehr nahe bei $\frac{\pi}{4}$ liegt}}
$\\ $
Wir verwenden den Startwert \[x_0=0.7853981633975\,.\] Dieser ist bereits auf 14 Dezimalstellen genau.
$ \\ $
$ \\ $
\underline{\bf{2.\;Erhöhung der Konvergenzordnung}}
$\\ $
Um sicher zustellen, dass die Konvergenzordnung voll zum Tragen kommt, wählen wir $P=2$. Hieraus ergibt sich die Konvergenzordnung $2\cdot 2+1=5$. Das bedeutet, dass sich die Anzahl der gültigen Dezimalstellen mit jedem Schritt näherungsweise verfünffacht. Wir benötigen dann 8 Iterationsschritte um mit diesem Startwert auf eine Genauigkeit von mindestens einer Million Stellen zu kommen. Wir erhalten die Fixpunktfunktion 
\[T\left(x\right)=x
-\frac
{\sin\left(x\right)-\cos\left(x\right)}
{\sin\left(x\right)+\cos\left(x\right)}
+\frac{1}{3}\left(\frac
{\sin\left(x\right)-\cos\left(x\right)}{\sin\left(x\right)+\cos\left(x\right)}\right)^3\]
und die Fixpunktiteration
\[x_{n+1}=x_{n}-\frac
{\sin\left(x_{n}\right)-\cos\left(x_{n}\right)}
{\sin\left(x_{n}\right)+\cos\left(x_{n}\right)}
+\frac{1}{3}\left(
\frac
{\sin\left(x_{n}\right)-\cos\left(x_{n}\right)}
{\sin\left(x_{n}\right)+\cos\left(x_{n}\right)}\right)^{3}\,,
\]
die wir in der Form
\[x_{n}=x_{n-1}-\frac
{\sin\left(x_{n-1}\right)-\cos\left(x_{n-1}\right)}
{\sin\left(x_{n-1}\right)+\cos\left(x_{n-1}\right)}
+\frac{1}{3}\left(
\frac
{\sin\left(x_{n-1}\right)-\cos\left(x_{n-1}\right)}
{\sin\left(x_{n-1}\right)+\cos\left(x_{n-1}\right)}\right)^{3}\,.
\]
verwenden werden. Der Index $n$ läuft dann synchron mit der Anzahl der Iterationen.
$\\ $
$\\ $
\underline{\bf{3.\;Schritt-haltend die Genauigkeit erhöhen}}
$\\ $
Die Fixpunktiteration (wie jede andere auch) ist selbst-korrigierend. Wir müssen nicht von Anfang an gleich mit einer Genauigkeit von einer Million Stellen rechnen. Wir können mit weitaus geringerer Genauigkeit loslegen.
$\\ $
$\\ $
Für die Berechnungen verwenden wir das Computer-Algebrasytem Maple, dass sich bereits bei den Ableitungen als nützlich erwiesen hat. Die Variable \textbf{Digits} gibt die Anzahl der Dezimalstellen an mit denen Maple floating-point Zahlen berechnet; gibt man zum Beispiel \textbf{Digits}=100 vor, so rechnet Maple mit 100 Stellen Genauigkeit.
$\\ $
$\\ $
Im ersten Schritt benötigen wir eine Genauigkeit von $14\cdot 5=70$ Stellen, also ist \textbf{Digits}=70. Im 2. Schritt liegt die benötigte Genauigkeit bei $14\cdot 5^2=350$ Stellen, also müssen wir \textbf{Digits}=350 setzen. So geht es weiter bis zum 7. Schritt, wo Maple mit $14\cdot 5^7=1093750$ Dezimalstellen rechnet. Jetzt haben wir die vorgegebene Anzahl von mindestens einer Million Stellen erreicht. Deshalb brauchen wir im 8. Schritt keine weitere Erhöhung der Stellenanzahl vornehmen.
\[\]
Wir sehen uns die beteiligten Funktionen
\[f\left(x\right)=\frac{\sin\left(x\right)-\cos\left(x\right)}{\sin\left(x\right)+\cos\left(x\right)};\quad
T\left(x\right)=x+\frac{\cos\left(x\right)-\sin\left(x\right)}{\sin\left(x\right)+\cos\left(x\right)}
-\frac{1}{3}\left(\frac{\sin\left(x\right)-\cos\left(x\right)}{\sin\left(x\right)+\cos\left(x\right)}\right)^3\]
zusammen mit der Winkelhalbierenden $y=x$ in einer Abbildung an:
\begin{center}
\fbox{\includegraphics[width=0.8\linewidth]{"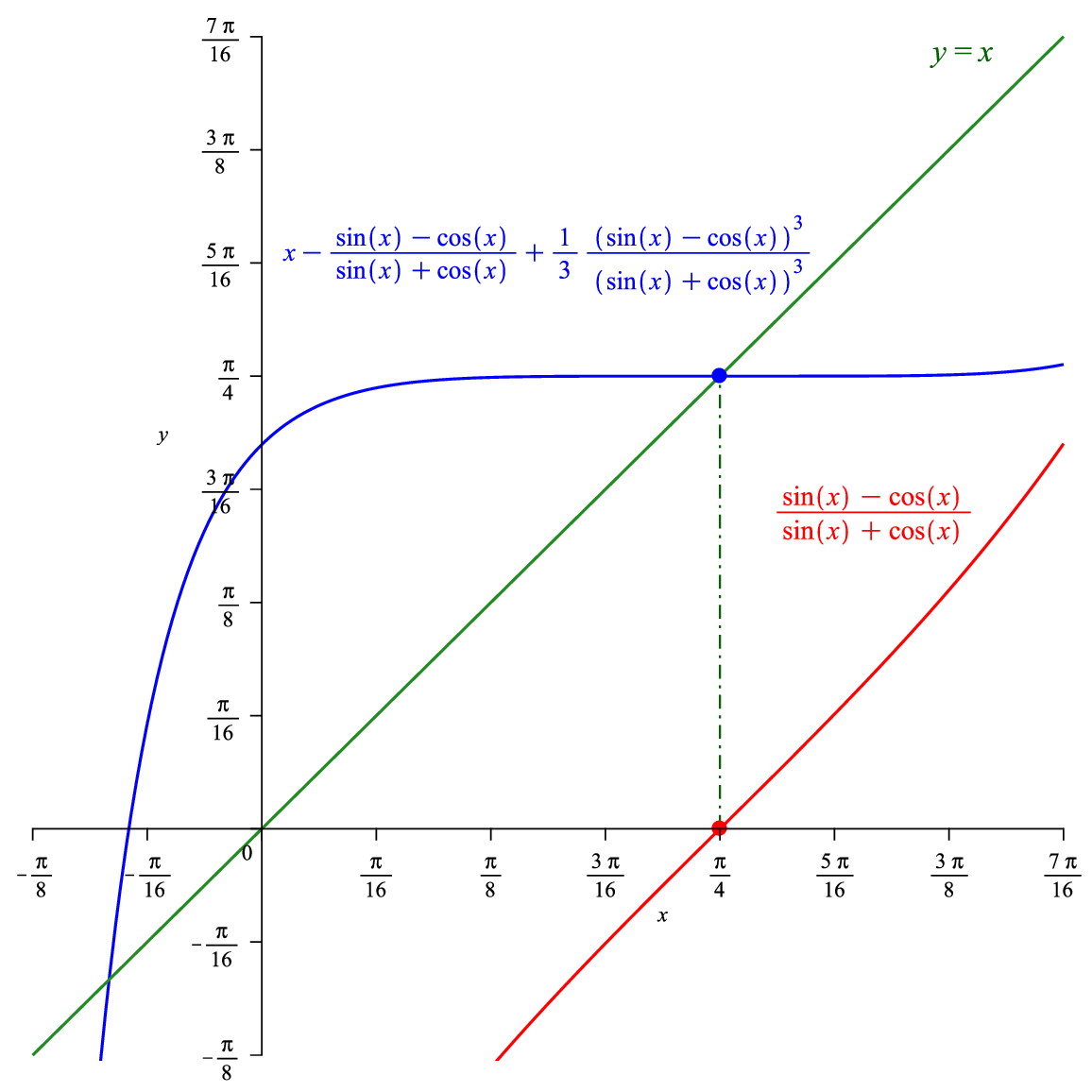"}}
\end{center}
\centerline{\textbf{\underline{Abbildung 1}}}
\[\]
In der nachfolgenden $\frac{\pi}{4}$ Berechnung haben wir die drei oben beschriebenen Optimierungsverfahren eingebaut. Wir starten mit \[x_0=0.7853981633975\] und berechnen acht Iterationen.
Zusätzlich geben wir zu jedem Schritt die Anzahl der \textbf{Digits} an, mit denen Maple gerade rechnet.
\[\underline{step\;\;1}\]
\[Digits=14\cdot 5=70\]
\[x_{{1}}
=x_0-{\frac
{\sin\left(x_{{0}}\right)-\cos\left(x_{{0}}\right)}
{\sin\left(x_{{0}}\right)+\cos\left(x_{{0}}\right)}}
+\frac{1}{3}\left(
{\frac{\sin\left(x_{{0}}\right)-\cos\left(x_{{0}}\right)}
{\sin\left(x_{{0}}\right)+\cos\left(x_{{0}}\right)}}\right)^{3}\]
\[x_{{1}}=0.7853981633974483096156608458198757210493\]
\[\left|x_{{1}}-x_{{0}}\right|={5.169038433915418012427895070765015622354\times10^{-14}}\]

\[\underline{step\;\;2}\]
\[Digits=14\cdot 5^2=350\]
\[x_{{2}}
=x_1-{\frac
{\sin\left(x_{{1}}\right)-\cos\left(x_{{1}}\right)}
{\sin\left(x_{{1}}\right)+\cos\left(x_{{1}}\right)}}
+\frac{1}{3}\left(
{\frac
{\sin\left(x_{{1}}\right)-\cos\left(x_{{1}}\right)}
{\sin\left(x_{{1}}\right)+\cos\left(x_{{1}}\right)}}\right)^{3}\]
\[x_{{2}}=0.7853981633974483096156608458198757210493\]
\[\left|x_{{2}}-x_{{1}}\right|={7.369842844775034299129366447073300446298\times10^{-68}}\]

\[\underline{step\;\;3}\]
\[Digits=14\cdot 5^3=1750\]
\[x_{{3}}
=x_2-{\frac
{\sin\left(x_{{2}}\right)-\cos\left(x_{{2}}\right)}
{\sin\left(x_{{2}}\right)+\cos\left(x_{{2}}\right)}}
+\frac{1}{3}\left(
{\frac
{\sin\left(x_{{2}}\right)-\cos\left(x_{{2}}\right)}
{\sin\left(x_{{2}}\right)+\cos\left(x_{{2}}\right)}}\right)^{3}\]
\[x_{{3}}=0.7853981633974483096156608458198757210493\]
\[\left|x_{{3}}-x_{{2}}\right|={4.348316332835180268524099971673673627949\times10^{-337}}\]

\[\underline{step\;\;4}\]
\[Digits=14\cdot 5^4=8750\]
\[x_{{4}}
=x_3-{\frac
{\sin\left(x_{{3}}\right)-\cos\left(x_{{3}}\right)}
{\sin\left(x_{{3}}\right)+\cos\left(x_{{3}}\right)}}
+\frac{1}{3}\left(
{\frac
{\sin\left(x_{{3}}\right)-\cos\left(x_{{3}}\right)}
{\sin\left(x_{{3}}\right)+\cos\left(x_{{3}}\right)}}\right)^{3}\]
\[x_{{4}}=0.7853981633974483096156608458198757210493\]
\[\left|x_{{4}}-x_{{3}}\right|={3.109106863680802672512752062886498044627\times10^{-1683}}\]

\[\underline{step\;\;5}\]
\[Digits=14\cdot 5^5=43750\]
\[x_{{5}}
=x_4-{\frac
{\sin\left(x_{{4}}\right)-\cos\left(x_{{4}}\right)}
{\sin\left(x_{{4}}\right)+\cos\left(x_{{4}}\right)}}
+\frac{1}{3}\left(
{\frac
{\sin\left(x_{{4}}\right)-\cos\left(x_{{4}}\right)}
{\sin\left(x_{{4}}\right)+\cos\left(x_{{4}}\right)}}\right)^{3}\]
\[x_{{5}}=0.7853981633974483096156608458198757210493\]
\[\left|x_{{5}}-x_{{4}}\right|={5.810429595425326132376821459074554489163\times10^{-8414}}\]

\[\underline{step\;\;6}\]
\[Digits=14\cdot 5^6=218750\]
\[x_{{6}}
=x_5-{\frac
{\sin\left(x_{{5}}\right)-\cos\left(x_{{5}}\right)}
{\sin\left(x_{{5}}\right)+\cos\left(x_{{5}}\right)}}
+\frac{1}{3}\left(
{\frac
{\sin\left(x_{{5}}\right)-\cos\left(x_{{5}}\right)}
{\sin\left(x_{{5}}\right)+\cos\left(x_{{5}}\right)}}\right)^{3}\]
\[x_{{6}}=0.7853981633974483096156608458198757210493\]
\[\left|x_{{6}}-x_{{5}}\right|={1.324558707074286945224127205648362976460\times10^{-42067}}\]

\[\]
\[\]

\[\underline{step\;\;7}\]
\[Digits=14\cdot 5^7=1093750\]
\[x_{{7}}
=x_6-{\frac
{\sin\left(x_{{6}}\right)-\cos\left(x_{{6}}\right)}
{\sin\left(x_{{6}}\right)+\cos\left(x_{{6}}\right)}}
+\frac{1}{3}\left(
{\frac
{\sin\left(x_{{6}}\right)-\cos\left(x_{{6}}\right)}
{\sin\left(x_{{6}}\right)+\cos\left(x_{{6}}\right)}}\right)^{3}\]
\[x_{{7}}=0.7853981633974483096156608458198757210493\]
\[\left|x_{{7}}-x_{{6}}\right|={8.154288164686902937605883297903247653927\times10^{-210336}}\]

\[\underline{step\;\;8}\]
\[Digits=1093750\]
\[x_{{8}}
=x_7-{\frac
{\sin\left(x_{{7}}\right)-\cos\left(x_{{7}}\right)}
{\sin\left(x_{{7}}\right)+\cos\left(x_{{7}}\right)}}
+\frac{1}{3}\left(
{\frac
{\sin\left(x_{{7}}\right)-\cos\left(x_{{7}}\right)}
{\sin\left(x_{{7}}\right)+\cos\left(x_{{7}}\right)}}\right)^{3}\]
\[x_{{8}}=0.7853981633974483096156608458198757210493\]
\[\left|x_{{8}}-x_{{7}}\right|={7.210415146280134406533351725774712443819\times10^{-1051677}}\]
$\\ $
Nach acht Iterationsschritten haben wir weit über eine  Million Dezimalstellen für $\frac{\pi}{4}$ erhalten. Genau genommen ergab sich eine Genauigkeit in der Größenordnung von 1051677 Dezimalstellen. An den negativen Potenzen der Differenzen können wir die Konvergenzordnung $5$ ablesen. Sie nehmen näherungsweise um den Faktor 5 zu.
$\\ $
$\\ $
Alle Berechnungen wurden auf einem Eigenbau-PC mit folgender Hardware ausgeführt: Motherboard ASUS PRIME A320M-K mit CPU AMD Ryzen 5 5600G 6 CORE 3.90-4.40 GHz und 32 GB RAM. Die verwendete Software war MAPLE 2025.2 von  Maplesoft, Waterloo Maple Inc. Die Software wurde ebenfalls für die Konvertierung der mathematischen Terme nach LaTeX verwendet.

\end{document}